\documentclass[final]{amsart}
\usepackage{caption}
\usepackage{subcaption}

\usepackage{mathabx,cite}
\usepackage[utf8]{inputenc} 
\usepackage[english]{babel}
\usepackage{amstext}
\usepackage{euscript}
\usepackage{bbm}
\usepackage{mathrsfs}
\usepackage{enumerate}
\usepackage{graphicx}
\usepackage{amscd}
\usepackage{verbatim}
\usepackage{amssymb}
\usepackage{amsthm}
\usepackage{amsmath}
\usepackage{amstext}
\usepackage{amsfonts}
\usepackage{latexsym}
\usepackage{mathtools} 
\usepackage{enumitem} 
\usepackage{accents} 
\usepackage[foot]{amsaddr} 

\usepackage[usenames,dvipsnames]{xcolor} 
\definecolor{dkblue}{RGB}{1,31,91} 
\usepackage[colorlinks=true, pdfstartview=FitV, linkcolor=dkblue, citecolor=dkblue, urlcolor=dkblue]{hyperref}

\usepackage{todonotes}

\usepackage[color]{showkeys}

\newcommand{\eqdef }{\overset{\mbox{\tiny{def}}}{=}}

\newcommand{\rth}{{\mathbb{R}^3}}

\theoremstyle{definition}
\newtheorem{theorem}{Theorem}

\newtheorem{corollary}[theorem]{Corollary}

\newtheorem{lemma}[theorem]{Lemma}
\newtheorem{proposition}[theorem]{Proposition}

\newtheorem{hypothesis}[theorem]{Hypothesis}
\newtheorem{remark}[theorem]{Remark}

\numberwithin{equation}{section}
\numberwithin{theorem}{section}
\numberwithin{definition}{section}

\begin{document}


\title[Temperature distribution of a body heated by radiation]{On the temperature distribution of a body heated by radiation}

\author[J. W. Jang]{Jin Woo Jang}
\address{Department of Mathematics, Pohang University of Science and Technology (POSTECH), Pohang, South Korea (37673). \href{mailto:jangjw@postech.ac.kr}{jangjw@postech.ac.kr} }

\author[J. J. L. Vel\'azquez]{Juan J. L. Vel\'azquez}
\address{Institute for Applied Mathematics, University of Bonn, 53115 Bonn, Germany. \href{mailto:velazquez@iam.uni-bonn.de}{velazquez@iam.uni-bonn.de} }

\begin{abstract}In this paper, we study the temperature distribution of a body when the heat is transmitted only by radiation. The heat transmitted by convection and conduction is ignored. We consider the stationary radiative transfer equation in the local thermodynamic equilibrium. We prove that the stationary radiative transfer equation coupled with the non-local temperature equation is well-posed in a generic case when emission-absorption or scattering of interacting radiation is considered. The emission-absorption and the scattering coefficients are assumed to be 
general and they can depend on the frequency of radiation. 
 We also establish an entropy production formula of the system, which is used to prove the uniqueness of solutions for an incoming radiation with constant temperature.
\end{abstract}
\thispagestyle{empty}

\maketitle
\tableofcontents

\section{Introduction}
In the paper, we will study the mathematical problem that describes the temperature of a body interacting with radiation.   This problem has been extensively studied  both in the physical and mathematical literature.  We consider a model that combines the photon radiation transfer equation with a non-local temperature equation of the body of materials (which could be solids, liquids, gases, or plasmas) interacting with radiation. More precisely, we focus on the problem of determining the local temperature of a body where heat transfer is due only to the radiation. In particular, the effects of heat conduction will be completely ignored in this paper. We are interested in the stationary distributions of the temperature that are due to the incoming radiation arriving to a body.
It turns out that the temperature distribution is determined by a non-local equation, which encodes the transfer of radiation at each different point where the absorption and the emission are taking place as well as the scattering phenomena: i.e., deflection of photons without change of their frequencies. In spite of the fact that this is a canonical problem in the theory of heat transfer processes, the well-posedness of the stationary problem has not been studied in detail in the mathematical literature to the best of our knowledge, although some of the existing literature that we will discuss later have considered some stationary related problems including also heat conduction.
In this paper, we will not restrict ourselves to so-called Grey approximations (i.e., the independence of the emission-absorption coefficients on the frequency), since we will consider the case where the scattering coefficient can also depend on the frequency. 

The problem that we consider contains only absorption,  emission, and scattering. One interesting aspect of the problem is that the presence of absorption, emission, and scattering is enough to determine the temperature of the material even in the absence of the Laplacian modeling heat conduction. The solutions that we obtain in the paper are stationary, non-equilibrium, constant flux solutions, which are characterized by an incoming flux of radiation interacts with the material and eventually the radiation escapes and yields an outgoing flux of radiation.

\subsection{Summary of relevant literature}

 The distribution of temperature of a body in which relevant part of the heat transfer takes place by means of radiation has been studied in several papers. The earliest studies of the interaction of radiation with gases are due to Compton and Milne  \cite{Compton, Milne}. Additional results were obtained by Holstein and Kenty \cite{PhysRev.42.823,Holstein}. In particular, in \cite{Holstein} it was noticed that in order to understand the transfer of heat by means of radiation the analysis of a non-local problem is required. A specific problem studied in 
\cite{spiegel1957smoothing} is
the evolution in time of the temperature of a bar in which the heat transport due only to the radiation and the temperature is close to a constant was computed using a linearized form of the equation as well as the Fourier transform method. See also \cite{oxenius,rutten1995radiative,mihalas2013foundations} for further general references.  


In \cite{Golse-Bardos, B3,golse2008radiative,golse1986generalized,golse1987milne,MR2600939,CRMECA_2022__350_S1_A12_0,B1,MR2076780}, the evolution in time of temperature of materials in which the radiation plays an important role has been considered. In most of these papers, the well-posedness of a time-dependent problem using semi-group theory, for instance the theory of $m$-accretive operators, has been obtained.   Notice that the existence of solutions to the time-dependent problem does not imply the existence of steady-states, as solutions may increase over time. Even with the presence of uniform bounds in time, the solutions can be oscillatory for long times.

Several of the mathematical papers  studying the heat transfer by means of radiation contain also terms describing heat conduction, both in time-dependent and stationary situations (cf. \cite{MR2478911,MR2600939,MR3412332, MR2076780, MR3797032, MR1866555, MR1608072,MR1471600}). In these cases, it is possible to use classical tools from elliptic and parabolic theory in order to gain regularity of solutions.
In \cite{MR1608072,Nouri2,MR2478911,MR1471600}, stationary distributions of temperature in the place of radiative transfer have been considered under several boundary conditions. The papers that are the closest to the problem that we consider in this paper are \cite{MR2478911,MR1608072, MR1471600}. In these papers, the equations contain a Laplacian $-k\Delta T$, that models the heat conductivity, and that plays a crucial role deriving a priori estimates. Moreover, in those papers the Grey approximation is assumed. In addition in \cite{MR1608072} a problem where the absorption of radiation and scattering inside the body play a crucial role is considered. In order to prove wellposedness for the stationary solution problem, the existence of suitable sub- and super-solutions is assumed without proof. On the other hand, in \cite{MR2478911,MR1608072, MR1471600}, the boundary conditions are different from the one that we assume in this paper. Specifically,  the papers \cite{MR2478911,MR1608072, MR1471600} assume an influx of heat in the energy equation while we do not impose any such conditions on the temperature at the boundary. On the contrary, in this paper, there is an incoming radiation coming from far away propagating through empty space (completely transparent material) and arriving to the body where it is absorbed and scattered.

In a different direction, in \cite{2208.04212,2109.10071,RSM} 
 solutions for the homogeneous Boltzmann equation combined with the radiative transfer equation have been studied. A recent result \cite{MR4519710} is about the distribution of the temperature in materials in which both heat conductivity and radiation play a role including numerical methods to approximate solutions of the system.

\subsection{Stationary radiative transfer equation in materials}
Throughout the rest of the paper, we will restrict ourselves to a physical situation where the radiation is interacting with materials. A general form of the stationary radiative transfer equation in materials including the excitation and the de-excitation of gas molecules as well as the scattering of photons can be written as
\begin{equation}\label{General rad eq}  n\cdot \nabla_x I_\nu = \alpha^a_\nu(T) B_\nu(T)- \alpha^a_\nu(T) I_\nu+  \alpha^s_\nu(T)G_\nu- \alpha^s_\nu(T) I_\nu,  
\end{equation}
under the assumptions that the gas molecules are in the local thermodynamic equilibrium (LTE) where  the terms $\alpha^a_\nu I_\nu $, $ \alpha^s_\nu I_\nu,\alpha^a_\nu B_\nu, \alpha^s_\nu G_\nu$ stand for absorption, loss term in the scattering, emission due to deexcitation of gas, and gain term in the scattering, respectively. Here $$B_\nu=B_\nu(T)\eqdef \frac{2h\nu^3}{c^2}\frac{1}{e^{\frac{h\nu}{kT}}-1}$$ is the Planck emission of a black-body and $I_\nu=I_\nu(x,n)$ stands for the intensity of the radiation in the frequency $\nu$ at position $x\in\Omega \in \rth$ with the direction $n\in \mathbb{S}^2$. Throughout the paper, we assume that $\partial \Omega$ is a $C^1$ boundary. The ``non-local" gain term of scattering with kernels is given by
\begin{equation}\label{scattering}\alpha^s_\nu(T)G_\nu = \alpha^s_\nu(T)\int_{\mathbb{S}^2} K(n,n')I_\nu(x,n')dn',\end{equation} with \begin{equation}\label{scattering assume}\int_{\mathbb{S}^2} K(n,n')dn=1.\end{equation} If the scattering is isotropic then it becomes simply $\alpha^s_\nu(T)I_\nu$ in \eqref{General rad eq}. 
Then we write the flux of radiation energy with frequency $\nu$ as $$\mathcal{F}_\nu=\mathcal{F}_\nu(x)=\int_{\mathbb{S}^2}nI_\nu(x,n)dn \in \rth.$$

A matter is regarded as being in a state of Local Thermodynamic Equilibrium (LTE) \cite[pp. 150]{oxenius} when microscopic mechanism, along with absorption and emission processes, occur with considerable rapidity at every point in the system.  The microscopic mechanism that tends to bring the distributions of states to equilibrium is fast enough to guarantee that the LTE assumption holds at any point of the material. If these conditions are not met, the matter is instead described as existing in a non-LTE state; see \cite{2109.10071, mihalas2013foundations,oxenius}. 
The class of models that we consider in this paper corresponds to the LTE situation. The local temperature distribution $T$ is well-defined at each point $x\in\Omega$, and the coefficients $\alpha_\nu^a$ and $\alpha^s_\nu$ can depend on the frequency of radiation $\nu$ and the local temperature distribution $T$ in general.

Throughout the rest of this paper, we will consider a generalized case where we take into account not just the emission-absorption process of radiation but also the scattering of photons in the stationary situation. To this end, we use a generic form of the stationary radiative transfer equation in materials \eqref{General rad eq}. 
Here, our radiation 
can have the continuous spectrum $\nu\in (0,\infty).$  Then 
we impose the following non-local temperature equation for the flux of radiation energy:
\begin{equation}\label{heat equation}\nabla_x\cdot \mathcal{F}(x)=0\text{ at any }x\in\Omega.\end{equation} Namely, in the LTE situation, we assume that the flux of radiation energy is divergence-free and the radiation energy is being conserved in each closed domain. This ansatz is shown to be true via \eqref{heat eq scattering only} in the case of pure scattering where the materials do not emit or absorb radiation, for example.  This will further be discussed in Section \ref{sec.1.5}.

\begin{figure}
    \centering
    \includegraphics[scale=0.09]{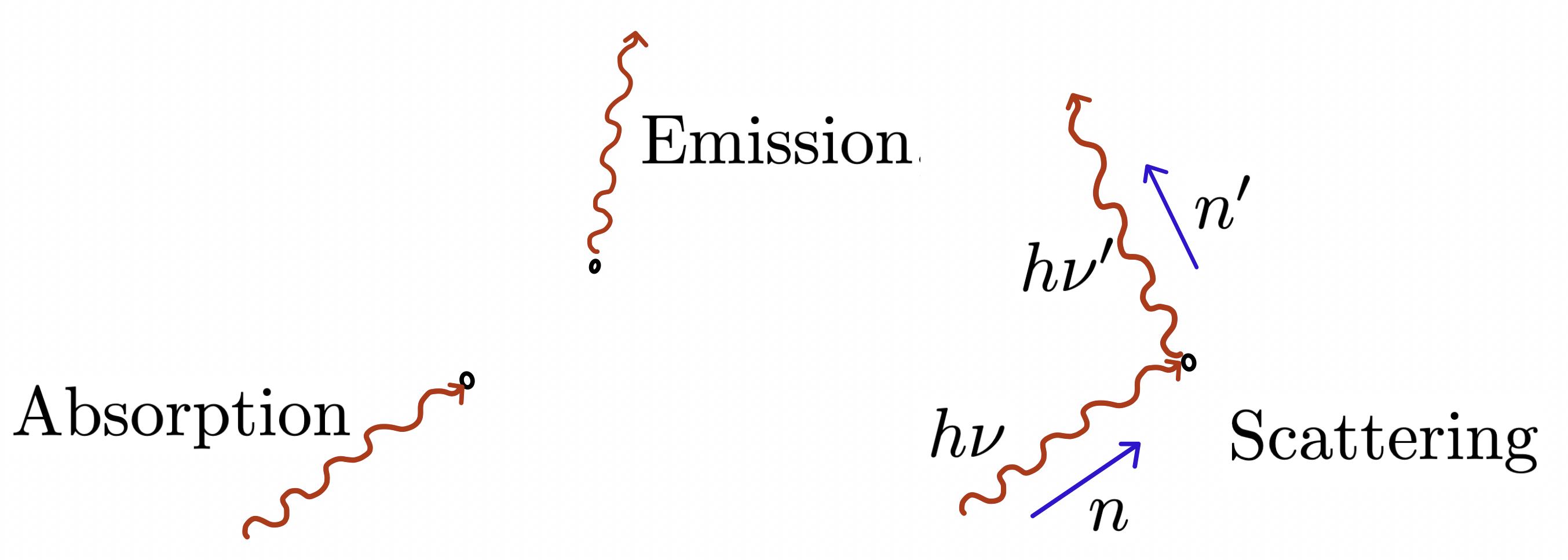}
    \caption{Absorption, Emission, and Scattering}
    \label{fig:my_label}
\end{figure}

\subsection{A mathematical problem and main assumptions} \label{sec.1.5}Conditions for the validity of \eqref{General rad eq} that we need are as in the following hypothesis. 
\begin{hypothesis}\label{modelling assumptions}Throughout the paper, we assume the following:
\begin{enumerate}
    \item The distributions of the material body are at LTE at each point of the material. More precisely, in the case of systems with possible several excited states at each point, they satisfy the Boltzmann ratio $n_u=e^{-\frac{\Delta E}{kT}}n_l$ with degeneracy equal to 1, where $n_u,n_l,\Delta E,k,$ and $T$ are number densities of upper and lower energy-states, the energy difference, the Boltzmann constant, and the local temperature, respectively.
    \item We consider the case that there is no displacement of the point of the body. In the case of the fluid, we consider the case where the fluid is at rest; i.e., the macroscopic velocity $u$ satisfies $u=0$. 
    \item The body of materials is in a convex bounded domain with $C^1$ boundary and there is a given profile of incoming radiation posed at the boundary.
\end{enumerate} 
\end{hypothesis}
\begin{remark} Here we make several comments related to the convexity assumption of the boundary of the domain $\Omega$.
\begin{enumerate}   
    \item We can obtain an equivalent problem by replacing the domain $\Omega$ by a larger ball $B_R(0)$ of radius $R>0$ containing $\Omega$ if this domain is convex. Doing that, we assume that $\alpha_\nu^a=\alpha_\nu^s=0$ in $B_R(0)\setminus \Omega.$ If the incoming boundary radiation at $\partial \Omega $, $n\cdot n_x<0$ is given by $g_\nu(x,n)$ (where $n_x$ is the outward normal vector at the boundary point $x\in\partial\Omega$),  we can extend then $I_\nu(x,n)$ along the lines $y=x-tn$, $t\ge 0$ (backwards) until reaching the boundary $\partial B_R(0)$. The convexity of $\Omega$ implies that this extension is well-defined. This gives the new boundary radiation $\tilde{g}_\nu(x,n)$ at $x\in \partial B_R(0),$ $n\cdot n_x<0$. We can then assume the boundary condition at $\partial B_R(0)$, with $\Omega \subset B_R(0)$. Notice that the coefficients $\alpha^a_\nu$ and $\alpha^s_\nu$ in this case is not constant and should depend also on the position $x$, but we think that the methods that we use in this paper will also allow to cover this case.
    \item If $\Omega$ is not convex, we cannot reduce the problem with boundary conditions $I_\nu(x,n)=g_\nu(x,n)$, $x\in\partial\Omega$, $n\cdot n_x<0$ to a problem with boundary values at the boundary of a larger ball $B_R(0).$ Indeed, in general the function $I_\nu(x,n)$ is not uniquely defined at $\partial B_R(0).$ (See Figure \ref{fig2} and \ref{fig3}.)
        \begin{figure}[h]
     \centering
     \captionsetup[subfigure]{labelformat=empty}
     \begin{subfigure}{0.3\textwidth}
         \centering
         \includegraphics[width=\textwidth]{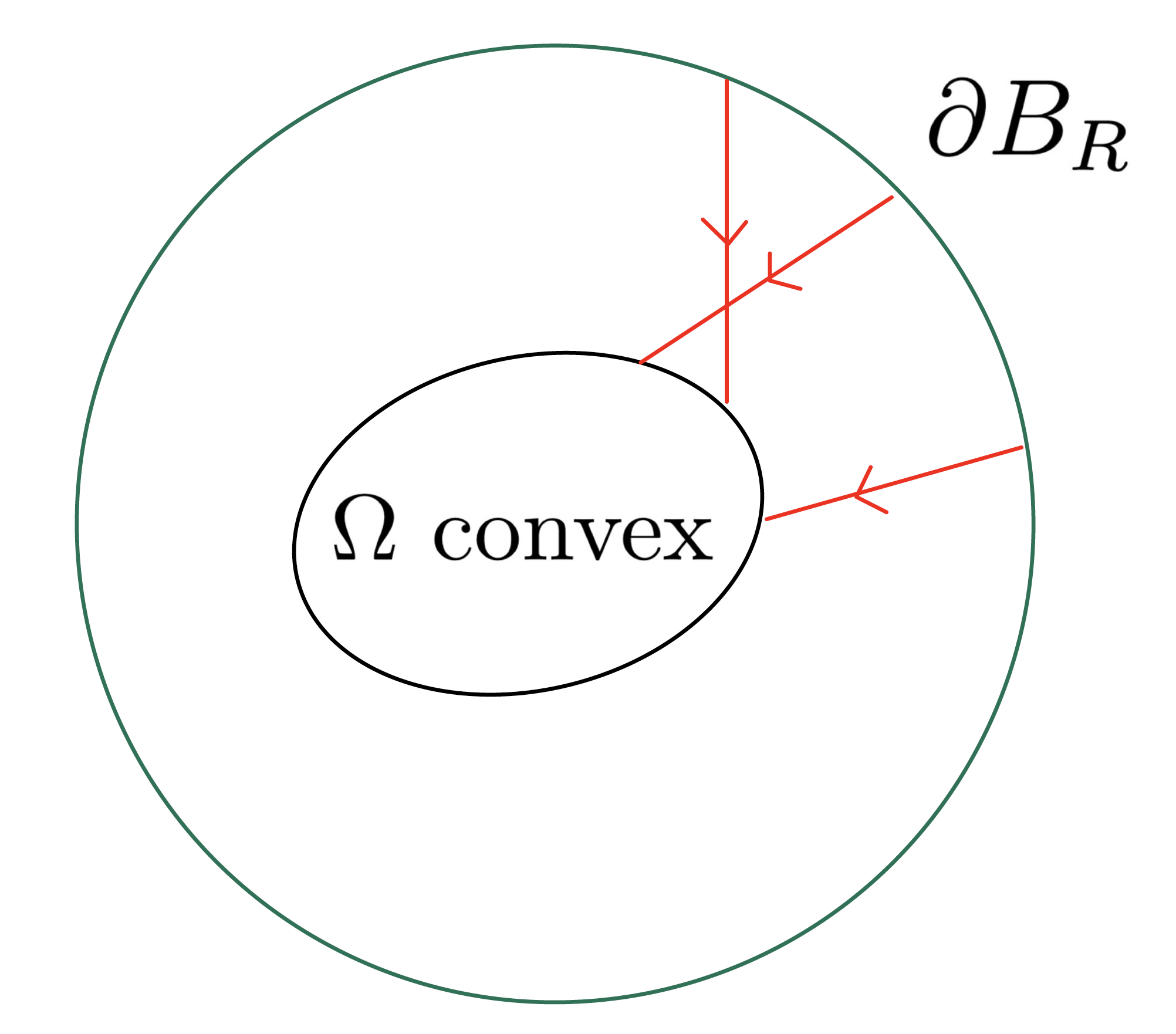}
         \caption{(a)}
         \label{fig2}
     \end{subfigure}\ \ \ \ \ \
     \begin{subfigure}{0.3\textwidth}
         \centering
         \includegraphics[width=\textwidth]{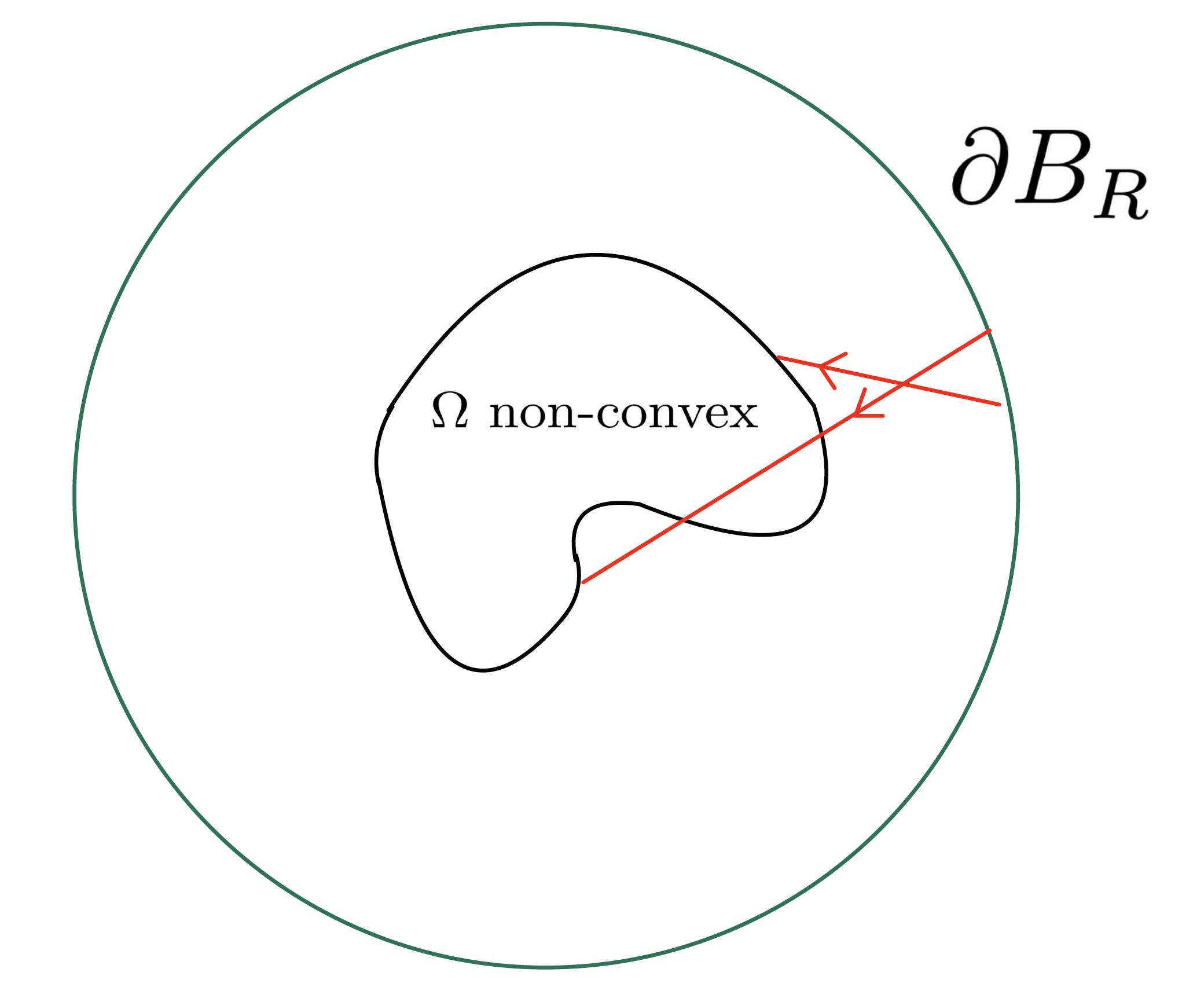}
         \caption{(b)}
         \label{fig3}
     \end{subfigure}
     \caption{(a) $I_\nu(x,n)$ defined uniquely in $\partial B_R(0)$.\newline  (b) $I_\nu(x,n)$ not uniquely defined in $\partial B_R(0)$.}
    \end{figure}
    
    To prove well-posedness for the boundary value problem with boundary values $I_\nu(x,n)=g_\nu(x,n)$, $\forall x\in\partial\Omega$, $n\cdot n_x<0$, $n\in \mathbb{S}^2$ in non-convex domains seems to require more sophisticated arguments than those used in this paper. The difficulty is that to determine $I_\nu(x,n)$ at some points of $\partial \Omega$, with $\Omega$ non-convex requires information about the boundary values of $I_\nu(x,n)$ at other points of $\partial \Omega$. (Figure \ref{fig4}) 
    \begin{figure}[h]
        \centering
        \includegraphics[scale=0.15]{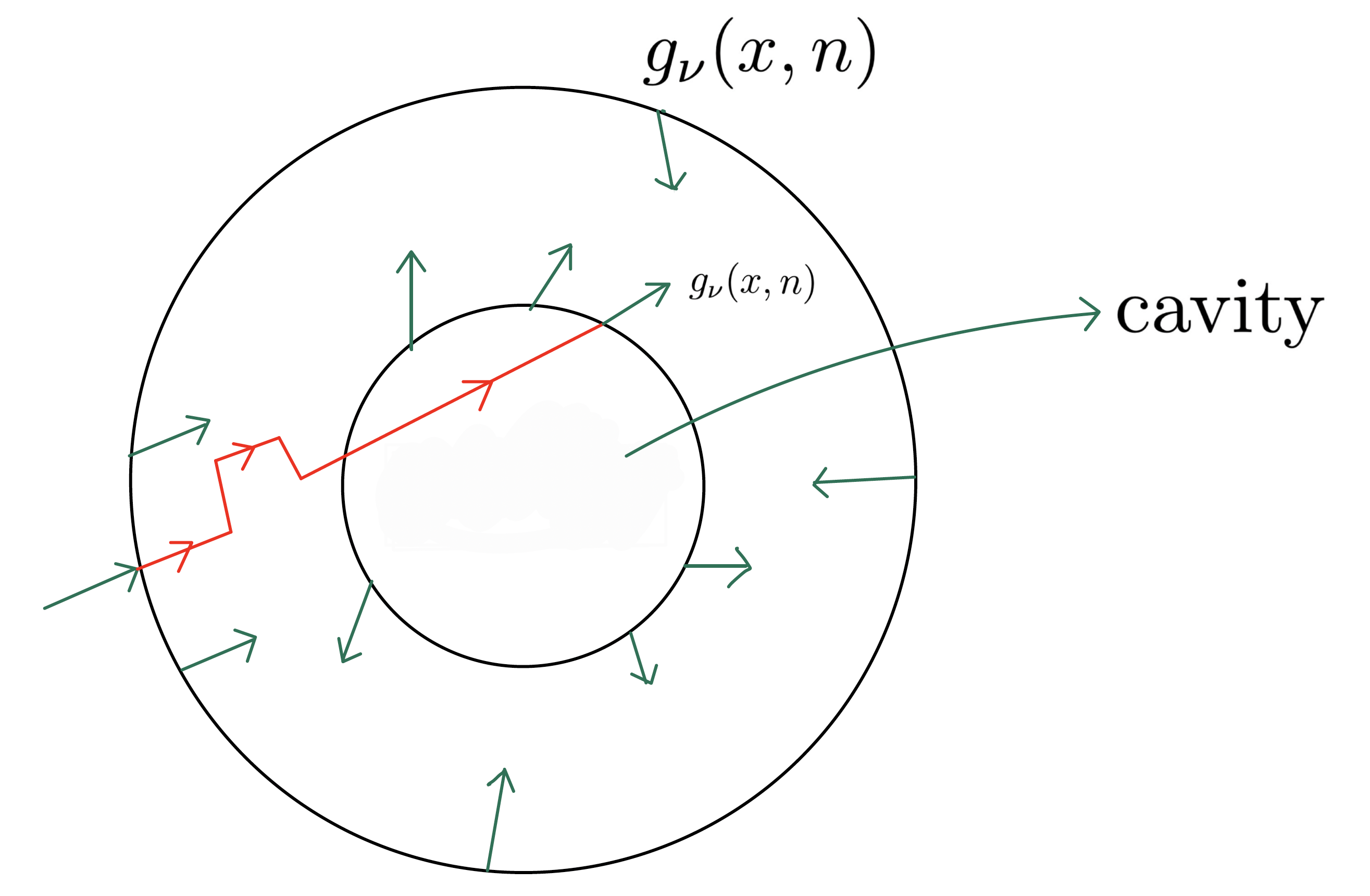}
        \caption{Connection between the values $I_\nu(x,n)$ at different boundary points in non-convex domain}
        \label{fig4}
    \end{figure}
    However, we can solve a different type of boundary problem. For instance, we can prescribe $I_\nu(x,n)$ at the boundary of the convexification of $\Omega$ (denoted as $\Omega_{conv}$). Alternatively, we could prescribe the incoming radiation at the boundary of a ball $B_R(0)$ containing $\Omega.$ See Figure \ref{fig5}.
 \begin{figure}[h]
        \centering
        \includegraphics[scale=0.15]{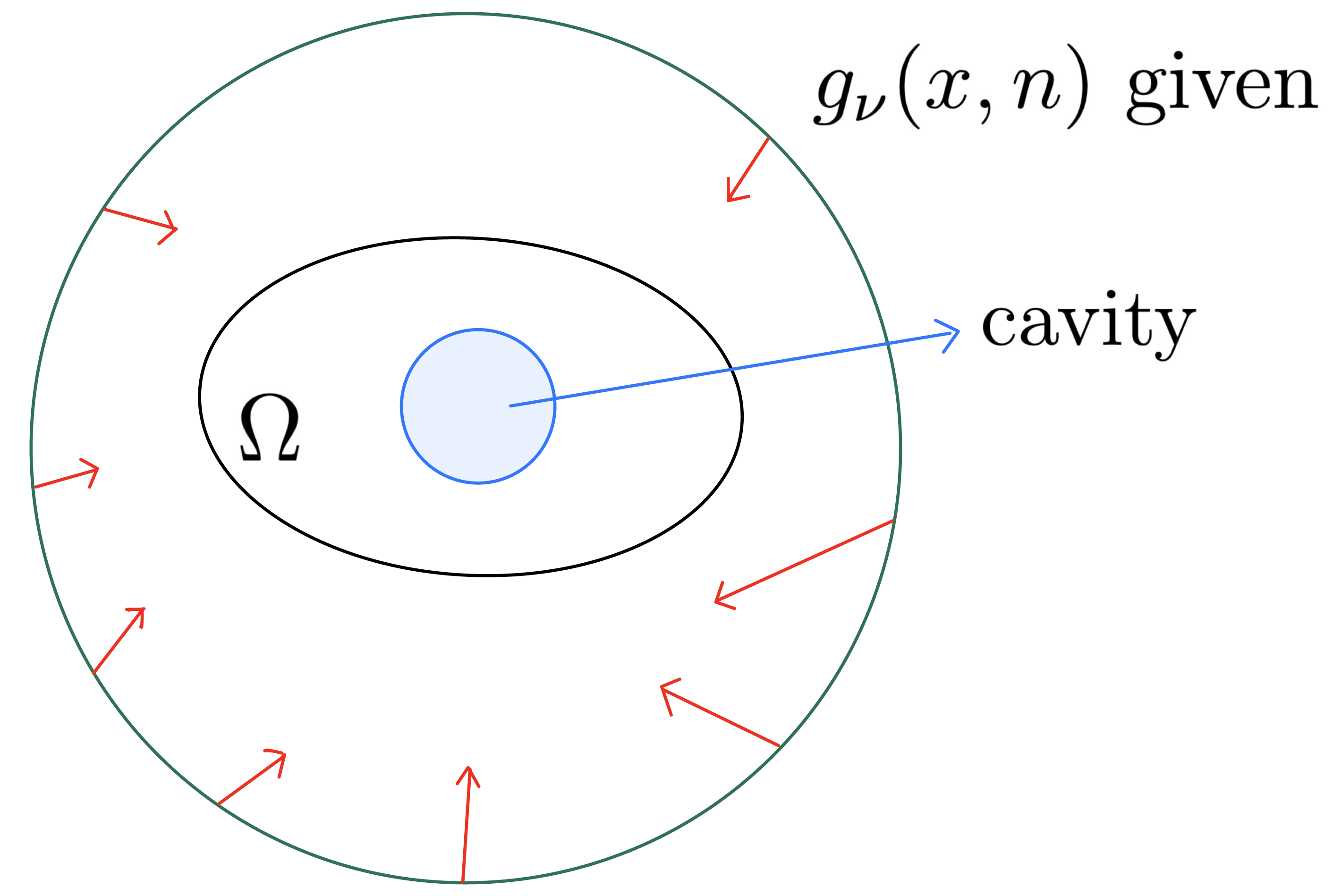}
        \caption{Notice that in this case, we do not prescribe $I_\nu$ in the whole boundary of $\Omega$.}
        \label{fig5}
    \end{figure}
\end{enumerate}
\end{remark}
Under the Hypothesis \ref{modelling assumptions}, we consider the radiative transfer equation \eqref{General rad eq} coupled with the non-local temperature equation \eqref{heat equation}. Here the scattering kernels further satisfy \eqref{scattering} and \eqref{scattering assume}. The assumption \eqref{scattering assume} on the kernel $K$ implies that the scattering does not modify the frequency.  The main assumption in this model above is that the non-elastic mechanisms yielding LTE in the gas molecules' distributions are extremely fast, and the scattering cannot modify much the Boltzmann ratio.  In the model, we also assume that $B_\nu(T)$ is the black-body radiation. 
Recall that we write the flux of radiation energy at $x$ as \begin{equation}\label{flux.radiation.energy}\mathcal{F}=\mathcal{F}(x)\eqdef \int_0^\infty d\nu \int_{\mathbb{S}^2}nI_\nu(x,n)dn.\end{equation}Then the divergence-free-flux condition \eqref{heat equation} for the radiation energy (i.e., the non-local heat equation)  gives
\begin{equation}\notag\nabla_x\cdot \mathcal{F}(x)=\int_0^\infty d\nu \int_{\mathbb{S}^2}n\cdot \nabla_xI_\nu(x,n)dn=0\text{ at any }x\in\Omega.\end{equation}The main goal of this paper is to study mathematically whether the temperature distribution $T$ at each point can be determined uniquely by \eqref{General rad eq}-\eqref{scattering assume} and a suitable boundary condition for the radiation at $\partial \Omega$, if we impose the divergence-free total-flux condition \eqref{heat equation}.  (See also \cite{2109.10071} for the conservation law.) 
Possible boundary conditions on the $C^1$ boundary $\partial\Omega$ of an open convex bounded domain $\Omega $ include the following incoming boundary condition
$$I_\nu(x,n)= g_\nu(x,n)\ge 0,\ \text{ for } x\in\partial\Omega\text{ and }n\cdot n_x<0,$$ where $n_x$ is the outward normal vector at the boundary point $x\in \partial \Omega.$   
\begin{remark}
     When the source $g_\nu$ does not depend on the position $x$ then it can also be understood that the radiation is coming to the body from infinity $(|x|= \infty).$ In this case, we obtain an incoming boundary condition for the intensity of the radiation $I_\nu(t,x,n)=g_\nu(n)$ on the body's boundary $x\in\partial \Omega $ for the incoming direction $n\cdot n_x<0.$
\end{remark}

\begin{remark}We will address the following problems (which have increasing level of difficulty):
\begin{enumerate}
    \item Materials where there is only scattering (Section \ref{sec.scattering}).
    \item Problems with only emission-absorption with a constant coefficient $\alpha_\nu^a=\alpha>0$ (Section \ref{sec.constant absorption}).
    \item Materials in which there is only emission-absorption with the coefficient $\alpha_\nu^a>0$ that depends only on the frequency $\nu$, 
    but not on the local temperature distribution $T(x)$ (Section \ref{sec.nu dependence}).
    \item Bodies with both scattering and emission-absorption with the coefficients $\alpha^s_\nu$ and $\alpha_\nu^a$ depend on the frequency $\nu$ 
    (Section \ref{sec.combined}).
\end{enumerate} Notice that in the second and the third cases, we use so-called the Grey approximations: i.e., $\alpha^a_\nu$ independent on $\nu$.

On the other hand, possible ways of modifying the problem that we do not consider in this paper are:
\begin{enumerate}
\item To consider problems with the absorption-emission and the scattering coefficients $\alpha^a_\nu$ and $\alpha^s_\nu$ also depend on the position $x$.
\item To include the absorption-emission coefficient that depends on the local temperature $T(x)$.
    \item To consider models that contain also the effect of heat conduction and/or convec tion in addition to the radiation. \item More general types of boundary conditions, for instance, boundaries reflecting part of the incoming radiation. \item Non-LTE situations for which we need to consider different temperatures for each of the microscopic species at each point of the material.
\end{enumerate}

\end{remark}

\subsection{Main strategy}
In all the different cases (with or without scattering term, and with or without the dependency of the coefficients on the radiation frequency), our main strategy for the proof of the well-posedness to the boundary-value problem is to reduce the stationary radiation equation and the ``divergence-free-radiation-flow" equation to a system of ``non-local" elliptic operators with external source. Then we will provide several techniques to tackle the non-local elliptic equations with sources. 

For the proof of the existence of solutions, especially when the situation is more complicated and the emission-absorption coefficient $\alpha^a$ depends on the frequency
, our goal is to deal with the stationary \textit{non-local heat equation} \eqref{heat equation} such that it is decomposed into the boundary contribution part and the Planck emission part. Then via taking several changes of variables that are closely associated with the characteristic trajectories of the emitted photon, we try to reduce the integro-differential equation into a form for which we can use the Banach fixed-point argument or the Schauder-type fixed-point argument. Here the sign of the boundary contribution is crucial and is subject to be proved as well.

The situation is further complicated if we consider the case when the radiation acting on materials undergoes the scattering as well as the emission and absorption as in Section \ref{sec.combined}, where one of the crucial ideas for the proof of existence is on the decomposition of the emission operator into the delta functions at each gas particle. This approach is motivated by the classical Green-function approach and this leads us to obtain the maximum principle and to continue with the Schauder-type fixed-point argument.

In particular, we will derive a uniqueness theorem using the entropy production argument and the variational principle in the specific case of constant temperature at the boundary when the contraction mapping principle would not work.


 \subsection{Outline}
 The rest of the paper is organized as follows.  In Section \ref{sec.scattering}, we first consider a special case with the pure scattering, which is analogous to so-called Lorentz gases, and prove the well-posedness of the boundary value problem when the gases do not emit or absorb radiation and the photons undergo the scattering only. In Section \ref{sec.emission only}, we consider the case when the radiation interacts with materials via emission and absorption, but without scattering. 
 We provide the proofs of existence and uniqueness of solutions in some particular subcases. In Section \ref{sec.entropy}, we give the formula for the entropy of the radiation and provide the entropy production formula. This will be a crucial step for the proof of the uniqueness in Section \ref{sec.uniqueness variation} where we make use of the entropy and a variational principle to prove the uniqueness of a solution with incoming radiation at equilibrium with constant temperature. 
 Lastly, in Section \ref{sec.combined} we consider the case when the radiation acting on materials undergoes the scattering as well as the emission and absorption.

\subsection{Notations}
 Throughout the paper, we use the following notation $A\approx B$ if there is a uniform constant $c>0$ such that 
 $\frac{1}{c}A\le B\le c A.$
\section{Well-posedness in the case of pure scattering}\label{sec.scattering}
In this section, we study the model with scattering only. This model is analogous to the case of Lorentz gases with randomly positioned scatterers. Under the assumption that there is no absorption or emission of radiation via the excitation and de-exitation of gas molecules, the model \eqref{General rad eq} reduces to 
\begin{equation}
    \label{onlyscattering eq}
      n\cdot \nabla_x I_\nu = - \alpha^s_\nu(T) I_\nu + \alpha^s_\nu(T)\int_{\mathbb{S}^2} K(n,n')I_\nu(x,n')dn',\end{equation} with \begin{equation}\label{eq.sec3.conservation}\int_{\mathbb{S}^2} K(n,n')dn=1.\end{equation} 
In this case, we remark that the temperature distribution $T$ does not play a crucial role because there is no absorption nor emission of photon; this can also be seen by \eqref{heat eq scattering only} below which is a direct consequence of \eqref{onlyscattering eq}. Then given any temperature distribution $T$, we find a distribution of radiation. This is one of the main differences with other cases with absorption only that will be introduced in Section \ref{sec.emission only} where the transfer of radiation depends crucially on the local temperature.      
\begin{remark}[Model with  Rayleigh scattering only] Examples of the model include the model with Rayleigh scattering only. 
In the model, one considers the following reaction
$$ A+h\nu(n_1) \leftrightarrow A+h\nu(n_2),$$ where $n_1,n_2\in \mathbb{S}^2$ stand for the directions of radiation before and after the scattering. The model can be written as
\begin{equation}\label{eq.scatteringonly}
    n\cdot \nabla_x I_\nu=-\alpha_\nu I_\nu +\alpha_\nu \int_{\mathbb{S}^2} K(n,n')I_\nu(x,n')dn',\end{equation} with $$\int_{\mathbb{S}^2} K(n,n')dn'=1.$$
 We remark here that the main difference of this model from other models that we consider in other sections of the paper is on the fact  that the local temperature distribution $T=T(x)$ is arbitrary in this scattering-only case; this is due to the fact that the solution $I_\nu(x,n)$ to \eqref{eq.scatteringonly} can be determined independently of the local temperature $T(x).$ Notice that for each solution $I_\nu(x,n)$ of \eqref{eq.scatteringonly} the divergence-free-flux condition $\nabla\cdot \mathcal{F}=0$ holds independently of $T(x)$. Thus the temperature distribution can be given arbitrarily. This will not happen in the case where there is also absorption and emission of radiation (cf. Section \ref{sec.emission only}).  Models analogous with \eqref{eq.scatteringonly} with only the scattering have been studied extensively in the mathematical literature, often in the context of neutron diffusion  \cite{MR533346,MR3324149,MR3686005}. This type of equations appears also in the theory of so-called Lorentz gases \cite{MR3356368,MR3324151}.
\end{remark}

Through the rest of the section, we prove the following well-posedness theorem:
\begin{theorem} \label{thm.scattering only}Fix any $\nu>0.$ 
Let $\Omega$ be a bounded domain with $C^1$ convex boundary  $\partial\Omega$. Suppose $\alpha^s_\nu(T)>0$ for any $T>0$. Suppose the local temperature distribution $T=T(x)$ and the scattering coefficient $\alpha^s_\nu(T)$ is given such that $\frac{1}{\alpha^s_\nu(T(\cdot))}$ is Lipschitz continuous. Consider the incoming boundary condition
\begin{equation}\label{incoming boundary}I_\nu(x,n)= g_\nu(x,n)\ge 0,\ \text{ for } x\in\partial\Omega\text{ and }n\cdot n_x<0,\end{equation} where $n_x$ is the outward normal vector at the boundary point $x\in \partial \Omega.$ Suppose that $g_\nu(x,n)\in L^\infty(\partial\Omega;L^1(\mathbb{S}^2\cap\{n\cdot n_x<0\})).$  Then the system of the stationary radiative transfer equation for pure scattering \eqref{onlyscattering eq}  has a unique solution $I_\nu(x,n)$ in $L^\infty(\Omega;L^1( \mathbb{S}^2) )$.
\end{theorem}

\begin{proof} First of all, we integrate \eqref{onlyscattering eq} with respect to $n\in\mathbb{S}^2$ and obtain
        \begin{multline}\label{heat eq scattering only}\nabla_x\cdot \mathcal{F}_\nu\eqdef \nabla_x\cdot \int_{\mathbb{S}^2}dn\ nI_\nu(x,n)\\= \int_{\mathbb{S}^2}dn\ \alpha^s_\nu(T)\left[\int_{\mathbb{S}^2} K(n,n')I_\nu(x,n')dn'-I_\nu(x,n)\right]\\
      = \alpha^s_\nu(T)\int_{\mathbb{S}^2}\ I_\nu(x,n')dn'- \alpha^s_\nu(T)\int_{\mathbb{S}^2}\ I_\nu(x,n)dn=0,\end{multline} for any $T$ and $\nu>0.$
\begin{figure}
    \centering
    \includegraphics[scale=0.12]{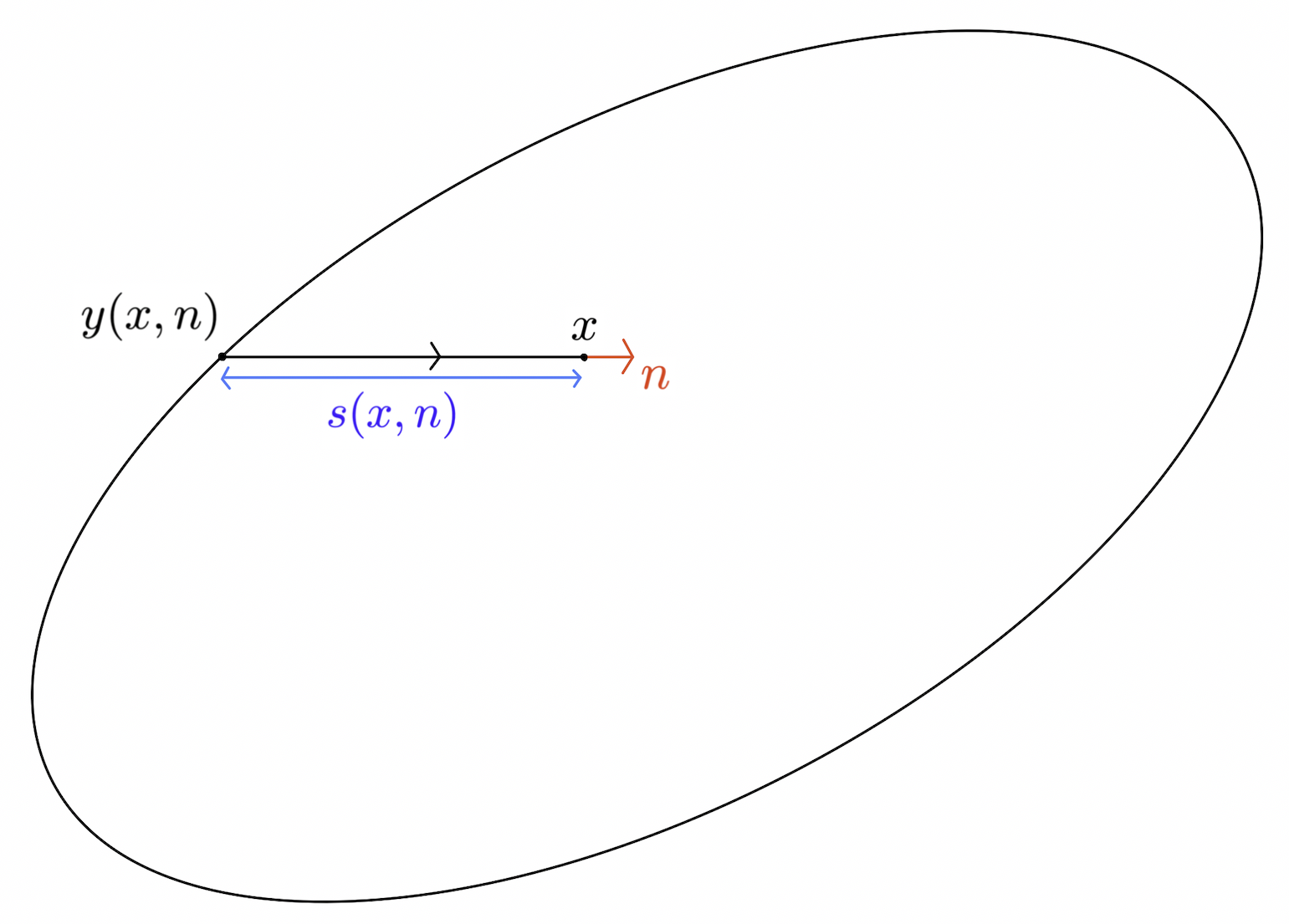}
    \caption{Trajectory in the direction $n$ and the boundary point $y(x,n)$}
    \label{fig:opt para}
\end{figure}
Now, for each $(x,n)\in \Omega \times \mathbb{S}^2$, we define a new coordinate system with new variables $(y,s)= (y(x,n),s(x,n))$. As introduced in Figure \ref{fig:opt para}, we define $y=y(x,n)\in \partial\Omega$ as the boundary point that intersects with the backward trajectory in direction $-n\in\mathbb{S}^2$ starting from $x\in \Omega.$ Then we define $s=s(x,n)\ge 0$ as the contour length by which we mean the traveling distance of the trajectory from $y$ to $x$ such that $x=y+sn.$ Then for the \textit{optical length} parameter $\tau_\nu\in [0,s],$ we define the characteristic trajectory $$X(\tau_\nu)=X(\tau_\nu;s,x,n),\ N(\tau_\nu)=N(\tau_\nu;s,x,n)=n,$$ such that $X(0)=y,\ X(s)=x,$ and 
$$\frac{dX(\tau_\nu)}{d\tau_\nu}= \frac{n}{\alpha^s_\nu(T(X(\tau_\nu)))}\text{ and } \frac{dN(\tau_\nu)}{d\tau_\nu}=0. $$ We note the characteristic trajectory, since $\frac{1}{\alpha^s_\nu(T(\cdot))}$ is Lipschitz continuous. Then we immediately observe that$$\int_y^x \alpha^s_\nu(T(X(\tau_\nu)) d(X(\tau_\nu))=\int_0^s n d\tau_\nu= ns,$$ and that 
\eqref{onlyscattering eq} now reads as
$$
   \frac{dI_\nu}{d\tau_\nu}=\frac{n}{\alpha^s_\nu(T)} \cdot \nabla_x I_\nu
    =\int_{\mathbb{S}^2} K(n,n')I_\nu(x,n')dn'-I_\nu(x,n).
$$   Thus we obtain\begin{equation}\label{eq.integro}
       I_\nu(x,n)=g_\nu(y,n)e^{-s}+\int_0^s e^{-(s-s')}ds'\int_{\mathbb{S}^2} K(n,n')I_\nu(y+s'n', n')dn'.
   \end{equation}
   Now define $$\mathcal{A}[I_\nu](x,n)\eqdef g_\nu(y,n)e^{-s}+A[I_\nu](x,n),$$ where
   $$A[I_\nu](x,n)\eqdef \int_0^{s(x,n)} e^{-(s(x,n)-s')}ds'\int_{\mathbb{S}^2} K(n,n')I_\nu(y+s'n', n')dn'.$$ For $(x,n)\in \Omega \times \mathbb{S}^2$, $A$ is well-defined. Then we observe
   \begin{multline*}\|A[I_\nu]\|_{L^\infty_xL^1_n} =\sup_{x\in\rth}\int_{\mathbb{S}^2}dn \int_0^{s(x,n)}ds' e^{-(s(x,n)-s')}\int_{\mathbb{S}^2} dn'K(n,n')I_\nu(y+s'n', n')\\\le \sup_{(x,n)\in\rth\times \mathbb{S}^2}\bigg(\int_0^{s(x,n)} e^{-(s(x,n)-s')}ds'\bigg)\int_{\mathbb{S}^2} K(n,n')dn\int_{\mathbb{S}^2}\|I_\nu(\cdot,n')\|_{L^\infty_x}dn'\\
   \le  (1-e^{-s_{max}})\|I_\nu\|_{L^\infty_xL^1_n},\end{multline*} since $\int_{\mathbb{S}^2} K(n,n')dn=1$. We denote the maximum diameter in the \textit{metric} of the optical path as $s_{max}$ then we have $0\le s(x,n)\le s_{max}$ and $\theta \eqdef (1-e^{-s_{max}})<1.$ Therefore, $\|A[I_\nu]\|_{L^\infty_xL^1_n}\le \theta \|I_\nu\|_{L^\infty_xL^1_n}$ with $\theta <1$. Thus, we can further show that $\mathcal{A}$ is a contraction mapping from $L^\infty(\Omega;L^1( \mathbb{S}^2) )\rightarrow L^\infty(\Omega;L^1( \mathbb{S}^2) )$, since $A$ (and hence $\mathcal{A}$) is linear in $I_\nu$ and we have
  $$
       \|\mathcal{A}[I_{\nu,1}]-\mathcal{A}[I_{\nu,2}]\|_{L^\infty_xL^1_n}=\|A[I_{\nu,1}]-A[I_{\nu,2}]\|_{L^\infty_xL^1_n}\le \theta \|I_{\nu,1}-I_{\nu,2}\|_{L^\infty_xL^1_n},
$$with $\theta<1$.
  Thus, by the Banach fixed-point theorem, for any $\nu>0$, we have the existence and the uniqueness of a fixed-point $I_\nu(x,n)$ in $L^\infty(\Omega;L^1( \mathbb{S}^2) )$, which is the solution of \eqref{eq.integro}. 
   \end{proof}
   This completes the proof of the wellposedness in the case when the heat is transferred only by the photons' scattering. In the next section, we consider another generic situation in which the heat is transferred only by the radiation without any scattering of photons.
   
   \section{Well-posedness for the model with emission-absorption only}\label{sec.emission only}Another physical situation that we consider next is the case with emission-absorption only. In this model, we assume that the scatterer's effect is negligible compared to the effect of the emission and the absorption of radiation from the molecules. 
   Then the stationary equation for this situation reads as
   \begin{equation}
       \label{eq. absorption only}
       n\cdot \nabla_x I_\nu =\alpha_\nu^a (T)(B_\nu(T)-I_\nu),
   \end{equation}
   where $I_\nu=I_\nu(x,n)$ is the intensity of the radiation of the frequency $\nu$ at $x\in \Omega$ and $n\in\mathbb{S}^2$, $\alpha^a_\nu(T)$ is the absorption coefficient (see (2.75) of \cite{rutten1995radiative}, for instance), and $B_\nu(T)$ is the Planck coefficient for the black-body radiation at temperature distribution $T$ in the LTE case:
   \begin{equation}
       \label{blackbody}
       B_\nu(T)=\frac{2h\nu^3}{c^2}\frac{1}{e^{\frac{h\nu}{kT}}-1}. 
   \end{equation}Note that $B_\nu(T)$ is monotonically increasing in $T$ for each $\nu$ and $B_\nu(T)=0$ is equivalent to $T=0.$ 
   By making a change of variables $\nu\to \zeta \eqdef\frac{h\nu}{kT} $, we obtain that 
   \begin{multline}\label{integration of blackbody}
       \int_0^\infty B_\nu(T)d\nu = \int_0^\infty\frac{2h\nu^3}{c^2}\frac{1}{e^{\frac{h\nu}{kT}}-1}d\nu=\int_0^\infty\frac{2hk^3T^3}{h^3c^2}\frac{\zeta^3}{e^{\zeta}-1}\frac{kT}{h}d\zeta\\=\frac{\pi^4}{15}\frac{2hk^3T^3}{h^3c^2}\frac{kT}{h}=\sigma T^4,
   \end{multline}
   where we define $
       \sigma\eqdef \frac{2\pi^4k^4}{15h^3c^2}>0.
 $

   In this section, we assume $\Omega$ is a convex domain. The coefficient $\alpha_\nu$ can be considered as the spectral lines for each $\nu$ or, more generally, the averages of these processes. There can also be more complicated dependence for the coefficient $\alpha_\nu$ if we have different \textit{bound states}. In systems with a given number of possible excited states and ionizations, it is also possible to obtain formula for $\alpha_\nu(T)$ assuming LTE using the Einstein relations and the cross-section for some of the collision processes.     As before, we assume the incoming boundary condition for the radiation \eqref{incoming boundary}.
    
    \subsection{The case with the constant absorption coefficient}\label{sec.constant absorption} The simplest case that we can consider is when the absorption coefficient $\alpha^a_\nu (T)=\alpha>0$ is independent of $\nu$ and $T.$ By rescaling $x$, we can assume without loss of generality that $\alpha_\nu = 1.$ Then \eqref{eq. absorption only} reduces to 
    \begin{equation}\label{eq.constant absorption}n\cdot \nabla_x I_\nu =B_\nu(T)-I_\nu.\end{equation}
    A similar boundary-value problem was considered in \cite{2109.10071}, but the model that we consider in this subsection is more general in the sense that this model now depends also on $\nu. $ Also, we consider the boundary profile \begin{equation}
        \label{boundary.1}  \begin{split}&I_\nu(x,n)=g_\nu(n)\ge 0,\  \text{for }x\in \partial\Omega,\ n\in \mathbb{S}^2,\ n\cdot n_x<0,\\
        &\int_{\mathbb{S}^2}\int_0^\infty g_\nu(n)d\nu dn=\int_{\mathbb{S}^2}g(n)dn <+\infty,\
        \end{split}
    \end{equation} such that the boundary condition depends only on the radiative direction $n \in \mathbb{S}^2$.  This boundary condition implies that the radiation is coming from a very far source of infinite distance.
    
    In the rest of this subsection, we prove the following well-posedness theorem:
    \begin{theorem}\label{thm3.1}
Let $\Omega$ be a bounded domain with $C^1$ convex boundary  $\partial\Omega$.   Then the system of the stationary radiative transfer equation with constant absorption coefficient  \eqref{eq.constant absorption} coupled with the conservation law \eqref{heat equation} and the incoming boundary condition \eqref{boundary.1}  admits a unique positive solution of the local temperature distribution $T=T(x)\in L^4(\Omega)$.
\end{theorem}
\begin{proof} We first 
define $y(x,n)$ and $s(x,n)$ the same as in Section \ref{sec.scattering}. Then the following characteristic trajectory $(X,N)$ along the optical path is well-defined where $\tau_\nu$ is the optical length parameter as introduced in Section \ref{sec.scattering}:$$X(\tau_\nu)=X(\tau_\nu;s,x,n),\ N(\tau_\nu)=N(\tau_\nu;s,x,n)=n,$$ such that $X(0)=y,\ X(s)=x,$ and 
$$\frac{dX(\tau_\nu)}{d\tau_\nu}= n\text{ and } \frac{dN(\tau_\nu)}{d\tau_\nu}=0. $$Then 
\eqref{eq.constant absorption} now reads as
$$
   \frac{dI_\nu}{d\tau_\nu}=n \cdot \nabla_x I_\nu
    =B_\nu(T)-I_\nu(x,n).
$$ 
By integrating this equation, we obtain
\begin{equation}\label{eq.integro1}
       I_\nu(x,n)=g_\nu(n)e^{-s(x,n)}+\int_0^{s(x,n)} e^{-(s(x,n)-\xi)} B_\nu(T(y(x,n)+\xi n))d\xi.
   \end{equation}
   Then, we have
   \begin{multline*}
       \mathcal{F}=\mathcal{F}(x)=\int_0^\infty d\nu \int_{\mathbb{S}^2}dn \ nI_\nu(x,n)\\
       =\int_0^\infty d\nu \int_{\mathbb{S}^2}dn \ ng_\nu(n)e^{-s(x,n)}
     \\  +\int_0^\infty d\nu \int_{\mathbb{S}^2}dn \ \int_0^{s(x,n)} e^{-(s(x,n)-\xi)}d\xi\  nB_\nu(T(y(x,n)+\xi n))\\
     =  \int_{\mathbb{S}^2}dn \ ng(n)e^{-s(x,n)}
     \\  + \int_{\mathbb{S}^2}dn \ \int_0^{s(x,n)} e^{-(s(x,n)-\xi)}d\xi\  n\sigma (T(y(x,n)+\xi n))^4\\
       \eqdef S(x)+\mathcal{U}[T](x),
   \end{multline*}where we recall that $
       \sigma\eqdef \frac{2\pi^4k^4}{15h^3c^2}>0.
 $
    Then by the \textit{non-local heat equation} \eqref{heat equation}, we obtain
   $$\nabla_x
   \cdot ( S+\mathcal{U}[T])(x)=0.$$
   In order to compute the divergence of $\mathcal{U}[T],$ we change the representation of the operator $\mathcal{U}$ to a convolution-type operator as follows (cf. \cite[Eq. (3.33) and (3.34)]{2109.10071}). By putting $w(x)\eqdef \sigma (T(x))^4$, we first observe that 
\begin{multline}\label{UT1}
        \mathcal{U}[T](x)=\int_{\mathbb{S}^2}n\left(\int_0^{s(x,n)} e^{-(s(x,n)-\xi)}w(y(x,n)+\xi n)d\xi\right)dn\\
        =\int_{\partial\Omega}dz\int_{\mathbb{S}^2}n\left(\int_0^{s(x,n)} e^{-(s(x,n)-\xi)}w(z+\xi n)\delta(z-y(x,n))d\xi\right)dn.
    \end{multline}
    We make a change of variables $\xi\mapsto \hat{\xi}=s(x,n)-\xi.$ Then we have\begin{multline*}
        \int_{\mathbb{S}^2}n\left(\int_0^{s(x,n)} e^{-(s(x,n)-\xi)}w(y(x,n)+\xi n)d\xi\right)dn\\
        =\int_{\partial\Omega}dz\int_{\mathbb{S}^2}n\left(\int_0^{s(x,n)} e^{-\hat{\xi}}w(z+(s-\hat{\xi}) n)\delta(z-y(x,n))d\hat{\xi}\right)dn.
    \end{multline*}
    Then we make another change of variables $(
    \hat{\xi},n)\mapsto \eta\eqdef x-\hat{\xi} n \in \Omega$. Since $x$ is independent of $\hat{\xi}$ and $n$, we obtain the Jacobian of the change of variables $J(\eta,n)\eqdef \left|\frac{\partial(\hat{\xi},n)}{\partial \eta}\right|=\frac{1}{|x-\eta|^2}$. Thus, for $n=n(x-\eta)=\frac{x-\eta}{|x-\eta|}$, we have
    \begin{multline}\label{UT2}
      \mathcal{U}[T](x)=\int_{\partial\Omega}dz\int_{\mathbb{S}^2}n\left(\int_0^{s(x,n)} e^{-\hat{\xi}}w(z+(s-\hat{\xi}) n)\delta(z-y(x,n))d\hat{\xi}\right)dn
\\   =\int_{\Omega}\frac{x-\eta}{|x-\eta|^3} e^{-|x-\eta|}w(\eta)d\eta,
    \end{multline}since $\hat{\xi}=(x-\eta)\cdot n= |x-\eta|.$
Note that $\mathcal{U}[T](x)$ is in a convolution form and the divergence in $x $ does not require any further regularity on $w$ and hence on $T$. By taking the divergence, we finally  have
   $$\nabla_x\cdot \mathcal{U}[T](x)=4\pi\sigma (T(x))^4-  \int_{\Omega} \frac{e^{-|x-\eta|}}{ |x-\eta|}\sigma(T(\eta))^4d\eta .$$ Therefore, we obtain a non-local equation for the temperature distribution $T$:
   $$\sigma(T(x))^4=  \int_{\Omega} \frac{e^{-|x-\eta|}}{4\pi |x-\eta|}\sigma (T(\eta))^4d\eta -\frac{1}{4\pi}(\nabla_x\cdot S)(x)\eqdef A(\sigma T^4).$$ This is a Dirichlet problem for a non-local elliptic operator. Note that $$\int_{\Omega} \frac{e^{-|x-\eta|}}{4\pi |x-\eta|}dx <\int_{\rth} \frac{e^{-|x-\eta|}}{4\pi |x-\eta|}dx=1.$$  Furthermore, the operator $A$ is linear in $\sigma T^4$ and in addition $A$ is a contractive mapping in $L^1(\Omega)$. Then by the Banach fixed-point theorem, we obtain the existence of a unique solution $\sigma T^4$ in $L^1(\Omega)$.
   
In order to have a physical solution (i.e., a local temperature distribution $T$ as a real-valued function), we must check if $\sigma T^4$ is real-valued and positive.  In order to make sure that the solution $\sigma T^4$ is real-valued and positive, it suffices to show that $-\frac{1}{4\pi}\nabla_x\cdot S>0.$ To this end, for any $x\in \Omega,$ we denote $y(x,n)$ as the intersection of the boundary $\partial\Omega$ and the backward characteristic trajectory $x-sn,$ for $s>0.$ Then we define the distance $|x-y|$ as $s(x,n)$. Then we observe that for any $x_0\in \Omega$,
$$y(x_0+\xi n,n)+s(x_0+\xi n,n)n = x_0+\xi n, $$ and hence $$y(x_0,n)+s(x_0+\xi n,n)n = x_0+\xi n.$$ By differentiating it with respect to $\xi$, we obtain 
$$\frac{d}{d\xi}(s(x_0+\xi n,n)n)=(\nabla_x s(x_0+\xi n,n)\cdot n)n=n.$$ Therefore, we have \begin{equation}\label{partialsn}n\cdot \nabla_xs(x_0,n) =1,\end{equation}
  and hence the definition of $S$ yields that
   \begin{multline*}
       -\nabla_x\cdot S= -\nabla_x\cdot \int_0^\infty d\nu \int_{\mathbb{S}^2}dn \ ng_\nu(n)e^{- s(x,n)}
       \\= \int_{\mathbb{S}^2}dn \  g(n)(n\cdot  \nabla s)e^{- s(x,n)}= \int_{\mathbb{S}^2}dn \  g(n)e^{- s(x,n)}> 0.
   \end{multline*}Therefore, the solution $\sigma T^4$ is real-valued and it is positive. Let $\sigma T^4 = a(x)$ for the positive real-valued solution $a\in L^1(\Omega).$ Then we have only two real-valued solutions $T$ as $T=\pm\left(\frac{a}{\sigma}\right)^{\frac{1}{4}}$. Thus, there is only one positive local temperature distribution $T=\left(\frac{a}{\sigma}\right)^{\frac{1}{4}}$ which solves the system. This completes the proof.
   \end{proof}
   This completes the case when the emission-absorption coefficient is constant. In the next subsection, we generalize the situation and consider the case when the emission-absorption coefficient now depends on the frequency.
   
   \subsection{The case when the absorption coefficient depends
   on \texorpdfstring{$\nu$}{}}\label{sec.nu dependence}
   Now we move onto the case when the absorbtion coefficient depends on the frequency  as well, $\alpha^a_\nu(T)=\alpha_\nu.$    In the rest of this subsection, we prove the following well-posedness theorem:
    \begin{theorem}\label{them 3.2}
Let $\Omega$ be a bounded domain with $C^1$ convex boundary  $\partial\Omega$.  Suppose that the emission-absorption coefficient $\alpha^a_\nu(T)=\alpha_\nu$ in \eqref{eq. absorption only} depends only on the radiation frequency.  Suppose further that 
$\int_0^\infty \alpha_\nu  d\nu<+\infty $ and the boundary profile $g_\nu$ satisfies $\alpha^2_\nu g_\nu(n)\in L^1([0,\infty)_\nu \times \mathbb{S}^2_n)$. Then the system of the stationary radiative transfer equation with the emission-absorption only \eqref{eq. absorption only} coupled with the conservation law \eqref{heat equation} and the incoming boundary condition \eqref{boundary.1} admits a solution of the local temperature distribution $T$ in the space $f(T)\in L^\infty_{\ge 0} (\Omega)$ (i.e., the space of non-negative $L^\infty$ functions on $\Omega$), where we define $$f(T(x))\eqdef \int_0^\infty  \alpha_\nu B_\nu(T(x))\ d\nu.$$ 
Moreover, the solution $f(T)\in L^\infty_{\ge 0} (\Omega)$ is unique. 
\end{theorem}
\begin{remark}
Note that $f$ (as well as $B_\nu$) is strictly monotone in $T$. Therefore, the uniqueness of $f(T)$ implies the uniqueness of $T$.
\end{remark}
\begin{proof}By modifying \eqref{eq.integro1} with the new type of coefficient $\alpha_\nu,$ we have
   \begin{equation}\label{eq.integro1new}
       I_\nu(x,n)=g_\nu(n)e^{-\alpha_\nu s(x,n)}+\alpha_\nu\int_0^{s(x,n)} e^{-\alpha_\nu(s(x,n)-\xi)} B_\nu(T(y(x,n)+\xi n))d\xi.
   \end{equation}
   Then, we have
   \begin{multline*}
     \mathcal{F}(x)
       =\int_0^\infty d\nu \int_{\mathbb{S}^2}dn \ ng_\nu(n)e^{-\alpha_\nu s(x,n)}
     \\  +\int_0^\infty d\nu\ \ \alpha_\nu \int_{\mathbb{S}^2}dn \ \int_0^{s(x,n)} e^{-\alpha_\nu(s(x,n)-\xi)} nB_\nu(T(y(x,n)+\xi n))d\xi\\
       \eqdef S(x)+\mathcal{U}[T](x).
   \end{multline*}Then by the \textit{non-local heat equation}  \eqref{heat equation}, we obtain
   $$\nabla_x
   \cdot ( S+\mathcal{U}[T])(x)=0.$$
   Now we compute $\nabla_x
   \cdot \mathcal{U}[T](x)$. By following similar changes of variables as in \eqref{UT1} - \eqref{UT2} with $w(x)\eqdef B_\nu(T(x))$ and an additional constant coefficient $\alpha_\nu$ in the exponent, we will write $\mathcal{U}[T](x)$ in a convolution-type form below. 

   We first put $w(x)\eqdef B_\nu(T(x))$ and observe that 
\begin{multline}\notag
        \mathcal{U}[T](x)=\int_0^\infty d\nu\ \ \alpha_\nu \int_{\mathbb{S}^2}n\left(\int_0^{s(x,n)} e^{-\alpha_\nu(s(x,n)-\xi)}w(y(x,n)+\xi n)d\xi\right)dn\\
        =\int_0^\infty d\nu\ \ \alpha_\nu \int_{\partial\Omega}dz\int_{\mathbb{S}^2}n\left(\int_0^{s(x,n)} e^{-\alpha_\nu(s(x,n)-\xi)}w(z+\xi n)\delta(z-y(x,n))d\xi\right)dn.
    \end{multline}
    We make a change of variables $\xi\mapsto \hat{\xi}=s(x,n)-\xi.$ Then we have\begin{multline*}
        \int_{\mathbb{S}^2}n\left(\int_0^{s(x,n)} e^{-\alpha_\nu(s(x,n)-\xi)}w(y(x,n)+\xi n)d\xi\right)dn\\
        =\int_{\partial\Omega}dz\int_{\mathbb{S}^2}n\left(\int_0^{s(x,n)} e^{-\alpha_\nu\hat{\xi}}w(z+(s-\hat{\xi}) n)\delta(z-y(x,n))d\hat{\xi}\right)dn.
    \end{multline*}
    Then we make another change of variables $(
    \hat{\xi},n)\mapsto \eta\eqdef x-\hat{\xi} n \in \Omega$. Since $x$ is independent of $\hat{\xi}$ and $n$, we obtain the Jacobian of the change of variables $J(\eta,n)\eqdef \left|\frac{\partial(\hat{\xi},n)}{\partial \eta}\right|=\frac{1}{|x-\eta|^2}$. Thus, for $n=n(x-\eta)=\frac{x-\eta}{|x-\eta|}$, we have
    \begin{multline}\notag
      \mathcal{U}[T](x)=\int_0^\infty d\nu\ \ \alpha_\nu \int_{\partial\Omega}dz\int_{\mathbb{S}^2}n\left(\int_0^{s(x,n)} e^{-\alpha_\nu\hat{\xi}}w(z+(s-\hat{\xi}) n)\delta(z-y(x,n))d\hat{\xi}\right)dn
\\   =\int_0^\infty d\nu\ \ \alpha_\nu \int_{\Omega}\frac{x-\eta}{|x-\eta|^3} e^{-\alpha_\nu|x-\eta|}w(\eta)d\eta,
    \end{multline}since $\hat{\xi}=(x-\eta)\cdot n= |x-\eta|.$
Note that $\mathcal{U}[T](x)$ is in a convolution form and the divergence in $x $ does not require any further regularity on $w$ and hence on $T$. 
  By taking the divergence on the convolution form, we conclude that
   \begin{multline*}
\nabla_x\cdot   \mathcal{U}[T](x)   
       =\int_0^\infty d\nu\ \ \alpha_\nu\int_\Omega d\eta\  B_\nu(T(\eta))\left(-\frac{\alpha_\nu}{|x-\eta|^2}e^{-\alpha_\nu |x-\eta|}+4\pi\delta(x-\eta)\right).
   \end{multline*}
   Thus we obtain the following equation:
\begin{multline} \label{eq for w}\nabla_x\cdot S -\int_0^\infty d\nu\ \ \int_\Omega d\eta\  B_\nu(T(\eta))\frac{\alpha_\nu^2}{|x-\eta|^2}e^{-\alpha_\nu |x-\eta|}\\+4\pi\int_0^\infty d\nu\ \alpha_\nu B_\nu(T(x))=0.\end{multline}
   Now define \begin{equation}\label{def.f.w}w(x)=f(T(x))\eqdef \int_0^\infty d\nu\ \alpha_\nu B_\nu(T(x)).\end{equation} By \eqref{blackbody} we observe that $w\ge0,$ $f$ is strictly monotone in $T$, and the inverse $f^{-1}$ exists so that 
   $T(x)=f^{-1}(w(x)).$  Now define the kernel \begin{equation}\label{def.K} K(x-\eta,w(\eta))=\tilde{K}(x-\eta, T(\eta))\eqdef \int_0^\infty d\nu \  B_\nu(T(\eta))\frac{\alpha_\nu^2}{|x-\eta|^2}e^{-\alpha_\nu |x-\eta|}.\end{equation} Then $\tilde{K}$ and $K$ are monotone in $T$ and $w$, respectively. Also note that
   \begin{equation} \label{w size}\int_\Omega dx\ K(x-y,w(y))<4\pi w(y),\end{equation} since $\Omega$ is bounded. 
 Note that, by \eqref{eq for w}, we have
   $$\nabla_x\cdot S = \int_\Omega d\eta\ K(x-\eta,w(\eta))-4\pi w(x).$$
   This is a second-type Fredholm non-linear integral equation. In order to discuss a Schauder-type fixed-point argument, we define the operator $J$ as
   $$J[w]\eqdef \frac{1}{4\pi}\int_\Omega d\eta\ K(x-\eta,w(\eta))-\frac{1}{4\pi}\nabla_x\cdot S.$$ We first show that $-\nabla_x\cdot S> 0$, since we observe \eqref{w size}. Since $n\cdot \nabla_xs=1$ by \eqref{partialsn}, the definition of $S$ yields that
   \begin{multline*}
       -\nabla_x\cdot S= -\nabla_x\cdot \int_0^\infty d\nu \int_{\mathbb{S}^2}dn \ ng_\nu(n)e^{-\alpha_\nu s(x,n)}
       \\=\int_0^\infty d\nu \int_{\mathbb{S}^2}dn \ \alpha_\nu g_\nu(n)(n\cdot  \nabla s)e^{-\alpha_\nu s(x,n)}\\=\int_0^\infty d\nu \int_{\mathbb{S}^2}dn \ \alpha_\nu g_\nu(n)e^{-\alpha_\nu s(x,n)}> 0.
   \end{multline*} Then this also yields that 
  $J:\{w\ge 0\}\to \{J[w]\ge 0\}$ and that $|\nabla_x\cdot S|<+\infty.$ Now suppose that $w\le L$ for some $L<0$ that is to be determined. Then we observe that
  \begin{multline*}
      J[w](x)=  \frac{1}{4\pi}\int_\Omega d\eta\ K(x-\eta,w(\eta))-\frac{1}{4\pi}\nabla_x\cdot S(x)\\
      \le  \frac{1}{4\pi}\int_\Omega d\eta\ K(x-\eta,L)+\frac{1}{4\pi}|\nabla_x\cdot S|
      \le \theta_\Omega L +\frac{1}{4\pi}\sup_{x\in \Omega}|\nabla_x\cdot S|,
  \end{multline*}for any $x\in \Omega$ where $\theta_\Omega\in(0,1)$ is a constant that depends only on $\Omega$ by \eqref{w size}. Choosing a sufficiently large $L>0$ such that  $$L>\frac{\sup_{x\in \Omega}|\nabla_x\cdot S|}{4\pi(1-\theta_\Omega)}, $$ we have $J[w]\le L. $ Thus we have that $J$ is a mapping from $$\mathcal{S}\eqdef \{w\in  L^\infty(\Omega): 0\le w(x)\le L,\text{ for any }x\in \Omega\},$$ to itself. In addition, $J$ is a compact mapping by the following argument below. In order to prove the compactness, we let $\{w_m\}_{m\in\mathbb{N}}$ be a bounded sequence in $\mathcal{S}$ with $\|w_m\|_{L^\infty}\le L.$ Recall the definition of the kernel $K$ in \eqref{def.K} and note that for every $\epsilon>0$ and $y\in [0,L]$ there exists a uniform $\delta\in (0,\epsilon)$ such that $|K(x,y)-K(x',y)|<\epsilon$ whenever $|x-x'|<\delta.$ Therefore, we have
  \begin{multline*}
      |J(w_m)(x)-J(w_m)(x')|\le \frac{1}{4\pi} \int_\Omega d\eta\ |K(x-\eta,w_m(\eta))-K(x'-\eta,w_m(\eta))|\\+\frac{1}{4\pi} \int_0^\infty d\nu \int_{\mathbb{S}^2}dn \ \alpha_\nu g_\nu(n)|e^{-\alpha_\nu s(x,n)}-e^{-\alpha_\nu s(x',n)}|\\
      \le \frac{\epsilon |\Omega|}{4\pi}+\frac{1}{4\pi}\int_0^\infty d\nu \int_{\mathbb{S}^2}dn \ \alpha_\nu g_\nu(n)|e^{-\alpha_\nu s(x,n)}-e^{-\alpha_\nu s(x',n)}|\\
      \le \frac{\epsilon |\Omega|}{4\pi}+\frac{1}{4\pi}\int_0^\infty d\nu \int_{\mathbb{S}^2}dn \ \alpha_\nu g_\nu(n)\int_0^1d\theta \ \alpha_\nu |e^{-\alpha_\nu s(x_{\theta},n) }||\nabla_x s(x_\theta,n)|  |x-y|\\
       \le \frac{\epsilon |\Omega|}{4\pi}+\frac{\epsilon C_{\Omega}}{4\pi}\int_0^\infty d\nu \int_{\mathbb{S}^2}dn \ \alpha_\nu^2 g_\nu(n),
  \end{multline*}for any $|x-x'|<\delta<\epsilon,$ where $x_{\theta}$ is defined as 
  $ x_{\theta}\eqdef (1-\theta_2)x+\theta_2 x',$ and we used the fundamental theorem of calculus and that $|\nabla_x s|$ is bounded by some constant $C_{\Omega}$ as the boundary $\partial \Omega$ is $C^1$.  
   Since $\alpha^2_\nu g_\nu(n)\in L^1([0,\infty)_\nu \times \mathbb{S}^2_n)$, this shows that $J(w_m)$ is an equicontinuous family of continuous functions on $\Omega$.  So there exists a subsequence $J(w_{m_k})$ that converges uniformly, and hence $J$ is compact because the image of a bounded sequence contains a convergent subsequence. Note that $\mathcal{S} $ is closed, bounded, convex, and non-empty. 
  Therefore, by the Schauder fixed-point theorem, $J$ has a fixed point and this completes the proof of the existence of a solution.
  
  For the proof of the uniqueness, we suppose that there are two solutions $w_1$ and $w_2$ in $ L^\infty(\Omega)$. 
  Then by taking the difference, we have
  $$w_1(x)-w_2(x)=\frac{1}{4\pi} \int_\Omega d\eta\ [K(x-\eta,w_1(\eta))-K(x-\eta,w_2(\eta))].$$
  Define a function 
  $$\Phi(t)\eqdef K(x-\eta,tw_2(\eta)+(1-t)w_1(\eta)),$$ such that $\Phi(0)=K(x-\eta,w_1(\eta))$ and $\Phi(1)=K(x-\eta,w_2(\eta))$. Then we observe that
  \begin{multline*}
      \Phi(1)-\Phi(0)=\int_0^1 \Phi'(t) dt = \int_0^1 \frac{\partial K}{\partial w}(x-\eta,tw_2(\eta)-(1-t)w_1(\eta))\cdot (w_2-w_1)(\eta)dt.
  \end{multline*}
   Thus, we have
   \begin{multline*}
       w_1(x)-w_2(x)=\int_\Omega (w_2-w_1)(\eta)d\eta\int_0^1 \frac{\partial K}{\partial w}(x-\eta,tw_2(\eta)-(1-t)w_1(\eta))\\
       =\frac{1}{4\pi }\int_\Omega d\eta (w_2-w_1)(\eta)\int_0^1 dt \int_0^\infty d\nu \frac{\partial B_\nu}{\partial T}\frac{\partial (f^{-1}(w))}{\partial w}(\eta) \alpha_\nu^2\frac{\exp(-\alpha_\nu|x-\eta|)}{|x-\eta|^2},
   \end{multline*}
   by the definition of $K$. Define 
   $$\mathcal{L}(x,w)(\eta)\eqdef \frac{1}{4\pi }\int_0^\infty d\nu \frac{\partial B_\nu}{\partial T}\frac{\partial (f^{-1}(w))}{\partial w}(\eta) \alpha_\nu^2\frac{\exp(-\alpha_\nu|x-\eta|)}{|x-\eta|^2}.$$
  By \eqref{blackbody}, we first have 
  $\frac{\partial B_\nu}{\partial T}\ge 0.$ Also, since $f^{-1}$ is monotone in $w$, we have $\frac{\partial (f^{-1}(w))}{\partial w}\ge 0.$ Therefore, $\mathcal{L}(x,w)(\eta)\ge 0. $
In addition, we have 
\begin{multline*}\int_\Omega \mathcal{L}(x,w)(\eta)dx< \int_{\mathbb{R}^3} \mathcal{L}(x,w)(\eta)dx \\
= \frac{1}{4\pi }\int_0^\infty d\nu  \ \frac{\partial B_\nu}{\partial T}\frac{\partial (f^{-1}(w))}{\partial w}(\eta) \alpha_\nu^2\int_{\mathbb{R}^3} dx\ \frac{\exp(-\alpha_\nu|x-\eta|)}{|x-\eta|^2}\\
=\int_0^\infty d\nu  \ \frac{\partial B_\nu}{\partial T}\frac{\partial (f^{-1}(w))}{\partial w}(\eta) \alpha_\nu,
\end{multline*}since 
$$\frac{\alpha_\nu}{4\pi}\int_{\mathbb{R}^3} dx\ \frac{\exp(-\alpha_\nu|x-\eta|)}{|x-\eta|^2}=1.$$
 On the other hand, by the inverse function theorem, we have 
   $$ \frac{\partial (f^{-1}(w))}{\partial w}= \frac{1}{\frac{\partial (f(T))}{\partial T}}.$$ Recall \eqref{def.f.w} that 
  $$\frac{\partial (f(T))}{\partial T}= \int_0^\infty d\nu\ \ \alpha_\nu \frac{\partial B_\nu}{\partial T},$$ and hence we have $$\int_\Omega \mathcal{L}(x,w)(\eta)dx< \int_0^\infty d\nu  \ \frac{\partial B_\nu}{\partial T}\frac{\partial (f^{-1}(w))}{\partial w}(\eta) \alpha_\nu=1.
$$
  Therefore, we finally obtain that 
   \begin{multline*}
       \| w_1-w_2\|_{L^1(\Omega)}=\int_0^1 dt\int_\Omega d\eta \ |w_2(\eta)-w_1(\eta)| \int_\Omega dx\ \mathcal{L}(x,w)(\eta)\\
       \le \theta \| w_1-w_2\|_{L^1(\Omega)},
   \end{multline*} for some constant $\theta <1$. Here note that $\Omega$ is bounded and hence $L^\infty(\Omega)\subset L^1(\Omega).$ Therefore,  $w_1=w_2$ and we conclude that the solution is unique.   \end{proof}

This completes the proof of the well-posedness in the case when the emission-absorption coefficient depends on the frequency. In the following sections, we introduce the radiation entropy production and  another approach to provide a uniqueness criterion for a generalized problem via the variational principle.

 \section{Radiation entropy production}\label{sec.entropy}
 In this section, we define the entropy density of the radiation and introduce the entropy production in the case without scattering. This will be used later in Section \ref{sec.uniqueness variation} for the proof of the uniqueness of stationary solutions.
 
We first rewrite the stationary radiative transfer equation without scattering as follows:
 \begin{equation}
     \label{timeindep.rte}
     n\cdot \nabla I_\nu = \varepsilon_\nu -\kappa_\nu I_\nu= \kappa_\nu(B_\nu(T)-I_\nu(T)),
 \end{equation}where $\kappa_\nu$ is the absorption coefficient and the source is given by the Planck distribution 
\begin{equation}\label{Planck}\frac{\varepsilon_\nu}{\kappa_\nu}=B_\nu(T)=\frac{2h\nu^3}{c^2}\frac{1}{e^{\frac{h\nu}{kT}}-1}.\end{equation} Notice that $\kappa_\nu$ can also depend on $x$ and we assume that $\kappa_\nu>0.$ Here let us define an analogous term $T_\nu(x,n)$ of the local temperature along each direction $n$ such that
 \begin{equation}
     \label{InuasBnuTnu}I_\nu(x,n)= B_\nu(T_\nu(x,n)),
 \end{equation} by means of the inverse of the function $B_\nu.$ This $T_\nu$ physically means the local temperature at which the radiation intensity $I_\nu$ can be written as the Planck distribution $B_\nu(T_\nu)$. Note that 
 \begin{equation}\label{Tnu as I nu}\frac{1}{T_\nu(x,n)}=\frac{k}{h\nu}\log\bigg(1+\frac{2h\nu^3}{c^2I_\nu(x,n)}\bigg),\end{equation} where $k$ is the Boltzmann constant.
 Then the entropy density of the radiation $s_{rad}=s_{rad}(x)$ is given by (see \cite[pp. 133, (5.1.35)]{oxenius}) 
 \begin{equation}
     \label{ent.den}
     s_{rad}(x) = \int_{\mathbb{S}^2}dn \int_0^\infty d\nu\ \ s_\nu (x,n),
 \end{equation}
where the local entropy density along each direction $n $ is defined as\begin{multline}
     \label{ent.den2}
     s_\nu (x,n) = k\frac{2\nu^2}{c^3}\bigg[\bigg(1+\frac{c^2}{2h\nu^3}I_\nu(x,n)\bigg)\log\bigg(1+\frac{c^2}{2h\nu^3}I_\nu(x,n)\bigg)\\-\frac{c^2}{2h\nu^3}I_\nu(x,n)\log\bigg(\frac{c^2}{2h\nu^3}I_\nu(x,n)\bigg) \bigg].
 \end{multline}Then by taking the gradient of $s_\nu$, we first obtain that 
 \begin{multline*}
     \nabla_x   s_\nu (x,n)=  k\frac{2\nu^2}{c^3}\nabla_x   \bigg[\bigg(1+\frac{c^2}{2h\nu^3}I_\nu(x,n)\bigg)\log\bigg(1+\frac{c^2}{2h\nu^3}I_\nu(x,n)\bigg)\\-\frac{c^2}{2h\nu^3}I_\nu(x,n)\log\bigg(\frac{c^2}{2h\nu^3}I_\nu(x,n)\bigg) \bigg]\\
     =k\frac{2\nu^2}{c^3} \bigg[\frac{c^2}{2h\nu^3}\nabla_x  I_\nu(x,n) +\frac{c^2}{2h\nu^3}\nabla_x  I_\nu(x,n) \log\bigg(1+\frac{c^2}{2h\nu^3}I_\nu(x,n)\bigg)\\-\frac{c^2}{2h\nu^3}\nabla_x  I_\nu(x,n) -\frac{c^2}{2h\nu^3}\nabla_x  I_\nu(x,n) \log\bigg(\frac{c^2}{2h\nu^3}I_\nu(x,n)\bigg)\bigg]\\
     =\frac{k}{h\nu c} \nabla_x  I_\nu(x,n) \log\bigg(\frac{1+\frac{c^2}{2h\nu^3}I_\nu(x,n)}{\frac{c^2}{2h\nu^3}I_\nu(x,n)}\bigg)\\
     =\frac{k}{h\nu c} \nabla_x  I_\nu(x,n) \log\bigg(1+ \frac{2h\nu^3}{c^2I_\nu(x,n)}\bigg)=\frac{1}{cT_\nu}\nabla_xI_\nu(x,n).
 \end{multline*}Then by plugging this into \eqref{timeindep.rte}, we can obtain that $s_\nu$ satisfies the entropy equation
 $$\nabla_x\cdot h_\nu = \hat{\sigma}_\nu,$$ where 
 $$h_\nu \eqdef cn s_\nu, \text{ and }
 \hat{\sigma}_\nu\eqdef \frac{1}{T_\nu(x,n)}[\varepsilon_\nu -\kappa_\nu I_\nu(x,n)]=\frac{\kappa_\nu }{T_\nu(x,n)}[B_\nu(T)-B_\nu(T_\nu)].$$
Here $h_\nu$ stands for the entropy flow density. 
 
 \subsection{Entropy production formula} Now we claim that the total entropy is being produced. More precisely, we introduce the following proposition on entropy production: 
 \begin{proposition}\label{entropy production proposition}Let $\Omega \subset \mathbb{R}^3$ be bounded. Assume that $\nabla_x\cdot\mathcal{F}(x)=0 $ at any $x\in \Omega,$ where $\mathcal{F}$ is the flux of radiation energy at $x$ defined in \eqref{flux.radiation.energy}.
Then we have \begin{multline*}
    \int_0^\infty d\nu\ \int_\Omega dx \int _{\mathbb{S}^2}dn\ \nabla_x\cdot h_\nu\\=\int_0^\infty d\nu\ \int_\Omega dx \int _{\mathbb{S}^2}dn\  \kappa_\nu\left(\frac{1 }{T_\nu(x,n)}-\frac{1 }{T(x)}\right)[B_\nu(T)-B_\nu(T_\nu)]\ge 0.
\end{multline*}  The equality holds if and only if $T_\nu(x,n)=T(x)$ for $\nu>0$, $x\in \Omega ,$ and $n\in\mathbb{S}^2$ a.e. 
 \end{proposition}\begin{remark}
 This proposition tells us that the entropy is being produced inside the body by the divergence theorem:
 $$ \int_0^\infty d\nu \int _{\mathbb{S}^2}dn\int_{\partial \Omega} h_\nu \cdot n_xdS_x \ge 0,$$ where $n_x$ is the outward normal vector at $x\in \partial \Omega.$ Namely, this implies that the total outgoing flow of the entropy leaving from the body is greater than or equal to the total incoming entropy into the body:\begin{multline*}
     \int_0^\infty d\nu \int _{\mathbb{S}^2\cap \{n\cdot n_x <0\}}dn\int_{\partial \Omega} h_\nu \cdot (-n_x)dS_x\le \int_0^\infty d\nu \int _{\mathbb{S}^2\cap \{n\cdot n_x >0\}}dn\int_{\partial \Omega} h_\nu \cdot n_xdS_x.
 \end{multline*}
 \end{remark}
 \begin{proof}[Proof of Proposition \ref{entropy production proposition}]
   By integrating \eqref{timeindep.rte} with respect to $n$ and using the divergence free flow assumption on $\mathcal{F}$, we have \begin{equation}\label{time dep. heat}
       \nabla_x\cdot \mathcal{F} =\int_0^\infty d\nu \int_{\mathbb{S}^2}dn\ \kappa_\nu [B_\nu(T)-B_\nu(T_\nu)]=0,
   \end{equation} for all $x\in \Omega$. 
   Then due to \eqref{time dep. heat} we observe that \begin{multline}\label{ent.eq1}
     \int_0^\infty d\nu\ \int_\Omega dx \int _{\mathbb{S}^2}dn\  \frac{\kappa_\nu }{T_\nu(x,n)}[B_\nu(T)-B_\nu(T_\nu)]\\
     = \int_0^\infty d\nu\ \int_\Omega dx \int _{\mathbb{S}^2}dn\  \kappa_\nu\bigg(\frac{1 }{T_\nu(x,n)}-\frac{1 }{T(x)}\bigg)[B_\nu(T)-B_\nu(T_\nu)]\\
     +\int_\Omega dx \frac{1 }{T(x)}\int_0^\infty d\nu \int _{\mathbb{S}^2}dn\  \kappa_\nu[B_\nu(T)-B_\nu(T_\nu)]\\
     =\int_0^\infty d\nu\ \int_\Omega dx \int _{\mathbb{S}^2}dn\  \kappa_\nu\bigg(\frac{1 }{T_\nu(x,n)}-\frac{1 }{T(x)}\bigg)[B_\nu(T)-B_\nu(T_\nu)].
   \end{multline}
   Then we note that the integrand is always non-negative, since \begin{multline}\label{ent.eq2}
       \bigg(\frac{1 }{T_\nu(x,n)}-\frac{1 }{T(x)}\bigg)[B_\nu(T)-B_\nu(T_\nu)]\\
       =\frac{2h_\nu^3}{c^2}(X_\nu-X)\left(\frac{1}{e^{h\nu X/k}-1}-\frac{1}{e^{h\nu X_\nu/k}-1}\right)\\
       =\frac{2h_\nu^3}{c^2}\frac{1}{(e^{h\nu X/k}-1)(e^{h\nu X}-1)}(X_\nu-X)\left(e^{h\nu X_\nu/k}-e^{h\nu X/k}\right)
       \ge 0,
   \end{multline}
   where we denote $X_\nu=\frac{1}{T_\nu}$ and $X=\frac{1}{T}.$ The equality holds if and only if $T_\nu(x,n)=T(x)$ for $\nu>0$, $x\in \Omega, $ and $n\in\mathbb{S}^2$ a.e.  This completes the proof of the entropy production.
 \end{proof}
As a corollary, we further obtain the local temperature distribution $T(x)$ is indeed constant everywhere if the equality holds. More precisely, we have the following proposition:\begin{proposition}\label{stationary entropy production proposition}
Let $\Omega\subset \mathbb{R}^3$ be bounded. Assume that $\nabla_x\cdot\mathcal{F}(x)=0 $ at any $x\in \Omega,$ where $\mathcal{F}$ is the flux of radiation energy at $x$ defined in \eqref{flux.radiation.energy}. In the stationary case, we have that the entropy is being produced; i.e., the total entropy outgoing flow is greater than or equal to the incoming entropy flow at the boundary on $\partial \Omega$: \begin{equation}\label{stationary entropy ineq}\int_0^\infty d\nu \int _{\mathbb{S}^2}dn\int_{\partial \Omega} dS_x \ n_x\cdot h_\nu \ge 0,\end{equation}where $n_x$ is the outward normal vector at the boundary point $x \in \partial \Omega$. Furthermore, if the equality holds in \eqref{stationary entropy ineq}, then the local temperature distribution $T(x)$ is constant everywhere.    
 \end{proposition}
 \begin{proof}
   The proof of \eqref{stationary entropy ineq} is a direct consequence of Proposition \ref{entropy production proposition}. Also, by \eqref{ent.eq1} and \eqref{ent.eq2}, the equality holds if and only if $T_\nu(x,n)=T(x) $ for $\nu>0$, $x\in \Omega, $ and $n\in\mathbb{S}^2$ a.e.  In this case, firstly for a $C^1$ solution $T(x)$ (and hence for $T_\nu$ and $I_\nu$ also in $C^1$), \eqref{timeindep.rte} and \eqref{InuasBnuTnu} imply for $\nu>0$, $x\in \Omega, $ and $n\in\mathbb{S}^2$ a.e. that\begin{multline*}
     0=  n\cdot \nabla_xI_\nu(x,n)=n\cdot \nabla_x B_\nu(T_\nu(x,n)) = n\cdot \nabla_x B_\nu(T(x))  \\= B'_\nu(T(x))(n\cdot \nabla_xT)(x).
   \end{multline*} Since $B'_\nu(T(x))\ne 0,$ we conclude that $\nabla_x T=0$ for $x\in\Omega$, and $T(x)=T_0$ for some constant temperature distribution $T_0.$ Now, if $T$ is just in $L^4$ or $L^\infty$ as in the existence theorems (Theorems \ref{thm3.1} and \ref{them 3.2}), we follow the same argument in the sense of distributions by multiplying compactly supported smooth test functions on both sides. 
   Then $T$ is a constant distribution and $T(x)=T_0$ for some constant temperature distribution $T_0.$ This completes the proof.  
 \end{proof}Finally, we conclude this section by stating a corollary as a direct consequence:
\begin{corollary}\label{cor.ent.prod.zero}Let $\Omega \subset \mathbb{R}^3$ be bounded. Assume that $\nabla_x\cdot\mathcal{F}(x)=0 $ at any $x\in \Omega,$ where $\mathcal{F}$ is the flux of radiation energy at $x$ defined in \eqref{flux.radiation.energy}. 
Suppose that the incoming radiation at the boundary is at equilibrium; i.e., suppose that $I_\nu(x,n)=B_\nu(T_0)$ for $x\in \partial\Omega$ and $n\cdot n_x<0$ where $n_x$ is the outward normal vector at the boundary point $x\in \partial\Omega$. Then there is a constant solution $T(x)=T_0$ minimizing the entropy and $I_\nu(x,n)=B_\nu(T_0)$ for $x\in\Omega$. The entropy production is zero if and only if the local temperature distribution $T(x)$ of the body $\Omega$ is equal to the constant $T_0$.  
\end{corollary}
Notice that in this corollary, we do not claim that the constant solution is unique. We will further discuss the uniqueness of constant solutions in the next section.

  \section{Uniqueness of the local temperature using a variational principle}
 \label{sec.uniqueness variation} In this section, we provide the proof of the uniqueness of the local temperature solution via looking at the entropy formula and a corresponding variational principle associated to the entropy. Throughout this section, we study the uniqueness of constant solutions to  \eqref{timeindep.rte} coupled with \eqref{heat equation} that maximizes the entropy in the stationary case, given that the incoming radiation is constant at the boundary $\partial \Omega$.

   We first recall \eqref{stationary entropy ineq} and denote the total outgoing entropy through the boundary $\partial \Omega$ as $\Phi_+$ where $\Phi_+$ is defined as
  \begin{equation}\label{totaloutgoingentropy}\Phi_+= \int_0^\infty d\nu  \int_{\partial \Omega}dS_x\int _{\{n\cdot n_x > 0\}\cap\mathbb{S}^2}dn(n\cdot n_x) \ cs_\nu.\end{equation} Also, denote the total outgoing radiation through the boundary $\partial \Omega$ as $i_{out}$ where $i_{out}$ is defined as
  \begin{equation}\label{totaloutgoingradiation}i_{out}= \int_0^\infty d\nu  \int_{\partial \Omega}dS_x\int _{\{n\cdot n_x > 0\}\cap\mathbb{S}^2}dn(n\cdot n_x) \ I_\nu.\end{equation} Similarly, we can define total incoming entropy through the boundary $\partial \Omega$ as $\Phi_-$ as
  \begin{equation}\label{totalincomingentropy}\Phi_-= \int_0^\infty d\nu  \int_{\partial \Omega}dS_x\int _{\{n\cdot n_x < 0\}\cap\mathbb{S}^2}dn(n\cdot n_x) \ cs_\nu,\end{equation} and the total incoming radiation  through the boundary $\partial \Omega$  $i_{in}$ as
  \begin{equation}\label{totalincomingradiation}i_{in}= \int_0^\infty d\nu  \int_{\partial \Omega}dS_x\int _{\{n\cdot n_x < 0\}\cap\mathbb{S}^2}dn(n\cdot n_x) \ I_\nu.\end{equation} Then we have the following lemma on the maximum outgoing entropy:
  \begin{lemma}\label{lemma.outent}
  Suppose that the total outgoing radiation $i_{out}>0.$ Then the total outgoing entropy $\Phi_+$ is at its   maximum if and only if $T_\nu(x,n)=T_0$ for some constant temperature distribution $T_0$ at the boundary $x\in \partial\Omega$ and $n\in\mathbb{S}^2$.
  \end{lemma}
  \begin{proof}
  We first recall \eqref{Tnu as I nu} and obtain in the stationary case that
  $$ 1+\frac{2h\nu^3 }{c^2I_\nu(x,n)}=e^{\frac{h\nu}{kT_\nu(x,n)}}.$$ Hence, we have
  $$1+\frac{c^2I_\nu(x,n)}{2h\nu^3 }=\frac{e^{\frac{h\nu}{kT_\nu(x,n)}}}{e^{\frac{h\nu}{kT_\nu(x,n)}}-1}.$$
  Then by recalling \eqref{ent.den2}, we rewrite $s_\nu(x,n)$ in terms of $T_\nu(x,n)$ as
  \begin{multline*}  s_\nu (x,n) = k\frac{2\nu^2}{c^3}\bigg[\bigg(\frac{e^{\frac{h\nu}{kT_\nu(x,n)}}}{e^{\frac{h\nu}{kT_\nu(x,n)}}-1}\bigg)\log\bigg(\frac{e^{\frac{h\nu}{kT_\nu(x,n)}}}{e^{\frac{h\nu}{kT_\nu(x,n)}}-1}\bigg)\\-\bigg(\frac{1}{e^{\frac{h\nu}{kT_\nu(x,n)}}-1}\bigg)\log\bigg(\frac{1}{e^{\frac{h\nu}{kT_\nu(x,n)}}-1}\bigg) \bigg]\\
  =k\frac{2\nu^2}{c^3}\bigg[\frac{h\nu}{kT_\nu(x,n)}\bigg(\frac{e^{\frac{h\nu}{kT_\nu(x,n)}}}{e^{\frac{h\nu}{kT_\nu(x,n)}}-1}\bigg)-\log\bigg({e^{\frac{h\nu}{kT_\nu(x,n)}}-1}\bigg) \bigg].\end{multline*} Now we use an alternative variable $Y_\nu=Y_\nu(x,n)=\frac{h\nu}{kT_\nu(x,n)}$ and write
  $$s_\nu (x,n)= k\frac{2\nu^2}{c^3}\bigg[Y_\nu\bigg(\frac{e^{Y_\nu}}{e^{Y_\nu}-1}\bigg)-\log\bigg({e^{Y_\nu}-1}\bigg) \bigg],$$ and $$I_\nu(x,n)=\frac{2h\nu^3}{c^2}\bigg(\frac{1}{e^{Y_\nu}-1}\bigg).$$
 Then by \eqref{totaloutgoingentropy} and \eqref{totaloutgoingradiation} we have 
 $$ \Phi_+= \int_0^\infty d\nu  \int_{\partial \Omega}dS_x\int _{\{n\cdot n_x > 0\}\cap\mathbb{S}^2}dn(n\cdot n_x) \ \frac{2k\nu^2}{c^2}\bigg[Y_\nu\bigg(\frac{e^{Y_\nu}}{e^{Y_\nu}-1}\bigg)-\log\bigg({e^{Y_\nu}-1}\bigg) \bigg],$$ and$$i_{out}= \int_0^\infty d\nu  \int_{\partial \Omega}dS_x\int _{\{n\cdot n_x > 0\}\cap\mathbb{S}^2}dn(n\cdot n_x) \ \frac{2h\nu^3}{c^2}\bigg(\frac{1}{e^{Y_\nu}-1}\bigg).$$
 In order to find the   maxima of $\Phi_+$ under the constraint $i_{out}>0$, we use the method of the Lagrange multiplier. Consider the variation of parameters $Y_\nu\to Y_\nu+s\varphi(x,v)$ for a real number $s\in\mathbb{R}$ for any given test function $\varphi.$ Then we have \begin{multline*} \frac{\partial\Phi_+}{\partial s}= \int_0^\infty d\nu  \int_{\partial \Omega}dS_x\int _{\{n\cdot n_x > 0\}\cap\mathbb{S}^2}dn(n\cdot n_x) \\\times \frac{2k\nu^2}{c^2}\frac{\partial}{\partial Y_\nu}\bigg[Y_\nu\bigg(\frac{e^{Y_\nu}}{e^{Y_\nu}-1}\bigg)-\log\bigg({e^{Y_\nu}-1}\bigg) \bigg]\varphi,\end{multline*} and$$\frac{\partial i_{out}}{\partial s}= \int_0^\infty d\nu  \int_{\partial \Omega}dS_x\int _{\{n\cdot n_x > 0\}\cap\mathbb{S}^2}dn(n\cdot n_x) \ \frac{2h\nu^3}{c^2}\frac{\partial}{\partial Y_\nu}\bigg(\frac{1}{e^{Y_\nu}-1}\bigg)\varphi.$$
 Hence define the Lagrangian function as
 $$\mathcal{L}(Y_\nu,\lambda)= \frac{2k\nu^2}{c^2}\bigg[Y_\nu\bigg(\frac{e^{Y_\nu}}{e^{Y_\nu}-1}\bigg)-\log\bigg({e^{Y_\nu}-1}\bigg) \bigg]-\lambda \frac{2h\nu^3}{c^2}\bigg(\frac{1}{e^{Y_\nu}-1}\bigg),$$ and it suffices to let $\frac{\partial \mathcal{L}}{\partial Y_\nu}=0$, since the choice of $\varphi$ is arbitrary.  
 By taking the derivative $\frac{\partial \mathcal{L}}{\partial Y_\nu},$ we observe 
 \begin{multline*}
     \frac{\partial \mathcal{L}}{\partial Y_\nu}= 1+\frac{1}{e^{Y_\nu}-1}-\frac{Y_\nu e^{Y_\nu}}{(e^{Y_\nu}-1)^2}-\frac{e^{Y_\nu}}{e^{Y_\nu}-1}+\lambda\frac{h\nu}{k} \frac{ e^{Y_\nu}}{(e^{Y_\nu}-1)^2}\\
     =\frac{ e^{Y_\nu}}{(e^{Y_\nu}-1)^2}\left(\lambda\frac{h\nu}{k}-Y_\nu\right).
 \end{multline*}This becomes zero if and only if $Y_\nu= \lambda\frac{h\nu}{k}$, which is constant. Thus, the outgoing entropy flow density is at its extreme point if and only if $T_\nu(x,n)=\frac{h\nu}{kY_\nu}= \frac{1}{\lambda}=T_0$ at the boundary $x\in \partial\Omega$ and $n\in\mathbb{S}^2$, for some constant $T_0,$ since $\varphi$ is arbitrary. 
 
 Now it suffices to prove that this extreme point is indeed a   maximizer of $\Phi_+$. For this, we make a change of variables $Y_\nu \to Z_\nu \eqdef \frac{1}{e^{Y_\nu}-1}.$ In this new variable, we have
 $$ \Phi_+= \int_0^\infty d\nu  \int_{\partial \Omega}dS_x\int _{\{n\cdot n_x > 0\}\cap\mathbb{S}^2}dn(n\cdot n_x) \ \frac{2k\nu^2}{c^2}\bigg[(Z_\nu+1)\log\bigg(\frac{Z_\nu+1}{Z_\nu}\bigg)+\log Z_\nu \bigg],$$ and$$i_{out}= \int_0^\infty d\nu  \int_{\partial \Omega}dS_x\int _{\{n\cdot n_x > 0\}\cap\mathbb{S}^2}dn(n\cdot n_x) \ \frac{2h\nu^3}{c^2}Z_\nu.$$ Denote the corresponding variables $Y_\nu,\ Z_\nu,$ and $T$ at the extreme point as $\bar{Y}_\nu,$ $\bar{Z}_\nu$, and $T_0$, respectively. Recall that $$\bar{Y}_\nu= \frac{h\nu}{kT_0},\text{ and }\bar{Z}_\nu=\frac{1}{e^{\frac{h\nu}{kT_0}}-1}.$$ We now want to show that the extreme point $\bar{Y}_\nu$ is indeed the maximizer of $\Phi_+.$ For this, we check if $\Phi_+(Y_\nu)- \Phi_+(\bar{Y}_\nu) $ is non-positive for any $Y_\nu$; namely, we check if the following term in $Z_\nu$ and $\bar{Z}_\nu$ is non-positive
 \begin{multline*}\Phi_+(Y_\nu)- \Phi_+(\bar{Y}_\nu)
 =\int_0^\infty d\nu  \int_{\partial \Omega}dS_x\int _{\{n\cdot n_x > 0\}\cap\mathbb{S}^2}dn(n\cdot n_x) \ \frac{2k\nu^2}{c^2}\\\times \bigg[(Z_\nu+1)\log(Z_\nu+1)-Z_\nu\log Z_\nu-(\bar{Z}_\nu+1)\log(\bar{Z}_\nu+1)+\bar{Z}_\nu\log \bar{Z}_\nu \bigg].\end{multline*}
 Then it suffices to check if $$G(Z_\nu)\eqdef (Z_\nu+1)\log(Z_\nu+1)-Z_\nu\log Z_\nu,$$ is concave. We take the first and the second derivatives and observe that
 $$G'(Z_\nu)= \log(Z_\nu+1)-\log Z_\nu,$$ and 
 $$G''(Z_\nu)= \frac{1}{Z_\nu+1}-\frac{1}{Z_\nu}=-\frac{1}{Z_\nu(Z_\nu+1)}<0.$$Hence $G$ is concave everywhere, and the extreme point $T_0$ is indeed a   maximizer. 
 This completes the proof.
  \end{proof}
  Similarly, we can also obtain the following lemma regarding the incoming radiation:
    \begin{lemma}\label{lemma.inc.ent}
  Suppose that the total incoming radiation satisfies $i_{in}>0.$ Then the total incoming entropy $\Phi_-$ is at its   maximum if and only if $T_\nu(x,n)=T_0$ for some constant temperature distribution $T_0$ at the boundary $x\in \partial\Omega$ and $n\in\mathbb{S}^2$.
  \end{lemma}\begin{proof}
  The proof is the same as the one for Lemma \ref{lemma.outent} except for the change that the integration with respect to the measure $dn$ is now on $\{n\cdot n_x<0\}\cap \mathbb{S}^2$ instead.   We omit the proof.
  \end{proof}
  As a result, we obtain the following proposition on the uniqueness of solutions for the case where the incoming radiation is at a thermal equilibrium:
  \begin{theorem}[Uniqueness of solutions]\label{unique thm 2} Let $\Omega$ be a bounded domain with $C^1$ convex boundary  $\partial\Omega$. Assume that $\nabla_x\cdot\mathcal{F}(x)=0 $ at any $x\in \Omega,$ where $\mathcal{F}$ is the flux of radiation energy at $x$ defined in \eqref{flux.radiation.energy}.  Suppose that the incoming radiation at the boundary $\partial \Omega$ is at a thermal equilibrium with a constant temperature; i.e., the incoming radiation profile $g_\nu$ satisfies $$g_\nu(x,n)=B_\nu(T_0), \ \text{for }x\in\partial\Omega,\ n\cdot n_x<0,$$ at the boundary with a constant $T_0>0$ where $B_\nu$ is the Planck distribution \eqref{Planck}. If the solution $T(x)$ is in $L^4$ or $L^\infty$ as in the existence theorems (Theorems \ref{thm3.1} and \ref{them 3.2}), then the temperature distribution $T$ of the body $\Omega$ is constant and is identical to the temperature of the incoming radiation. 
  \end{theorem}
  \begin{proof}
  By Proposition \ref{entropy production proposition}, Proposition \ref{stationary entropy production proposition} and Lemma \ref{lemma.inc.ent}, we observe that the incoming radiation is at its maximum entropy for the constant temperature at the boundary. At the boundary, the incoming intensity $I_\nu(x,n)$ is then determined uniquely by the constant temperature (say $T_0$), as $T_\nu(x,n)=T_0$ for $\nu>0$, $x\in \partial \Omega$ and $n\cdot n_x<0$. Also, by the conservation law $\nabla\cdot \mathcal{F}=0$ and the divergence theorem, we obtain that
  \begin{equation}\label{in out radiation}\int_0^\infty d\nu \int_{\mathbb{S}^2}dn\int_{\partial \Omega} dS_x (n\cdot n_x) I_\nu(x,n)=0.\end{equation} This implies $i_{out}=i_{in}$. Then, for $i_{out}$ given, the total outgoing entropy is at its maximum if and only if the temperature distribution $T(x,n)$ for $x\in \partial \Omega$ and $n\cdot n_x>0$ is constant. Note that by \eqref{in out radiation} the constant temperature for the outgoing radiation must be equal to $T_0$, if it is constant. Then, we notice that the outgoing radiation intensity function $I_\nu(x,n)$ defined via $T_\nu(x,n)=T_0$ provides the same value for the total outgoing entropy $\Phi_+$ as the total incoming entropy $\Phi_-$ by the definitions \eqref{totaloutgoingentropy}, \eqref{totalincomingentropy}, and \eqref{ent.den2}.  This implies that the production of entropy is zero and Proposition \ref{stationary entropy production proposition} and Corollary \ref{cor.ent.prod.zero} further implies that the temperature of the body $\Omega$ is constant and is identical to the temperature of the radiation. This gives the uniqueness of solutions for incoming radiation with constant temperature.  
  \end{proof}
  
  This completes the proof of the uniqueness via the variational method. In the next section, we finally consider the generic combined case in which we consider both the scattering and the  emission-absorption terms in the equation.
  
   \section{Combined case of scattering and emission-absorption}\label{sec.combined}In this section, we study the combined case with both scattering and absorption-emission terms in the equation. We recall the full equation \eqref{General rad eq} and \eqref{heat equation}
 and have the system of non-local heat equations 
 \begin{equation}\label{full system}
     \begin{split}
     n\cdot \nabla_x I_\nu &= \alpha^a_\nu 
     B_\nu(T)+ \alpha^s_\nu
     \int_{\mathbb{S}^2}K_\nu(n,n')I_\nu(x,n')dn'- (\alpha^a_\nu
     +\alpha^s_\nu
     ) I_\nu,  \\
         \nabla_x \cdot \mathcal{F}&\eqdef \int_0^\infty d\nu \int_{\mathbb{S}^2}n\cdot \nabla_x I_\nu(x,n)dn=0.
     \end{split}
 \end{equation}Here the scattering coefficient $\alpha_\nu^s
 \ge 0$ and the emission-absorption coefficient $\alpha^a_\nu
 \ge 0$ depend on 
 $\nu$. The scattering kernel $K_\nu$ satisfies 
 \begin{equation}\label{scatteringone}\int_{\mathbb{S}^2}K_\nu(n,n')dn=1.\end{equation}
Note that the \textit{non-local heat equation} $\nabla_x \cdot \mathcal{F}=0$ can be rewritten as a \textit{non-local} elliptic equation for $T(x).$
Regarding the system \eqref{full system} we have the following existence theorem. 
\begin{theorem}[Existence]\label{thm.full.exist}
Let $\Omega$ be a bounded domain with $C^1$ convex boundary  $\partial\Omega$.  Define \begin{equation}\label{deff}f(T(x))=\int_0^\infty \alpha_\nu^a
B_\nu(T(x))d\nu, \end{equation} and suppose that $f$ is increasing. 
Then a solution $f(T(\cdot))\in L^\infty(\Omega)$ to \eqref{full system} coupled with the non-local temperature equation \eqref{heat equation} and the incoming boundary condition \eqref{incoming boundary} exists. 
\end{theorem}
\begin{remark}
An almost similar proof for the theorem works for the generalized case when the absorption-emission and the scattering coefficients $\alpha^a_\nu$ and $\alpha^s_\nu$ depend also on $x\in\Omega$ if they are uniformly continuous in $x\in\Omega$ and there exist uniform functions  $\alpha_1=\alpha_1(\nu)\ge 0$ and $\alpha_2=\alpha_2(\nu)\ge0 $ such that $\alpha^a_\nu(x)\approx \alpha_1$ and $\alpha^s_\nu(x)\approx \alpha_2$ in $\Omega$. We omit the proof for this more general case for the sake of brevity, as the proof for the uniform continuity $\int_{\mathbb{S}^2}dn\ H_\nu(\cdot,n)$ simply involves more difference terms in \eqref{diff1}-\eqref{diff5}.
\end{remark}
\begin{proof}[Proof of Theorem \ref{thm.full.exist}]
In order to study the existence of the system \eqref{full system} coupled with the incoming boundary condition \eqref{incoming boundary}, we define and study a Green-type function $\tilde{I}_\nu(x,n;x_0)$ for each $x_0\in \Omega$ by means of the solution of 
\begin{equation}\label{greentype eq}
     \begin{split}
     &n\cdot \nabla_x \tilde{I}_\nu+ (\alpha^a_\nu
     +\alpha^s_\nu
     ) \tilde{I}_\nu - \alpha^s_\nu
     \int_{\mathbb{S}^2}K_\nu(n,n')\tilde{I}_\nu(x,n')dn'=\frac{1}{4\pi }\delta(x-x_0),
     \end{split}
 \end{equation}for $x\in\Omega$ and $n\in \mathbb{S}^2.$ In this direction, the full radiation intensity $I_\nu(x,n)$ can be retrieved by the relationship
 \begin{equation}\label{inu and itildenu} I_\nu(x,n)= \int_\Omega \tilde{I}_\nu (x,n;x_0)\alpha_\nu^a
 B_\nu(T(x_0))dx_0.\end{equation} The idea behind this formulation of $\tilde{I}_\nu(x,n;x_0)$ is that we only count emitted photons coming from $x=x_0.$
 If we define the corresponding flux $\tilde{\mathcal{F}}$ as $$
 \tilde{\mathcal{F}}_\nu(x;x_0)\eqdef \int_{\mathbb{S}^2}n\tilde{I}_\nu(x,n;x_0)dn, $$ 
 we can retrieve the flux $\mathcal{F}_\nu$ via \begin{equation}\label{Fnu and Ftildenu} \mathcal{F}_\nu(x)= \int_\Omega \tilde{\mathcal{F}}_\nu(x;x_0)\alpha_\nu^a
 B_\nu(T(x_0))dx_0.\end{equation}
 Integrating \eqref{greentype eq} with respect to $n\in\mathbb{S}^2$, we obtain 
$$
     \nabla_x\cdot \tilde{\mathcal{F}}_\nu+ \alpha^a_\nu
     \int_{\mathbb{S}^2}
     \tilde{I}_\nu dn=\delta(x-x_0),
 $$ by \eqref{scatteringone}. By defining 
\begin{equation}\label{Gdef}G_\nu(x,x_0)\eqdef \alpha^a_\nu
\int_{\mathbb{S}^2}\tilde{I}_\nu(x,n;x_0) dn,\end{equation}  we have
 \begin{equation}
     \label{flow ansatz}
     \nabla_x \cdot \tilde{\mathcal{F}}_\nu(x;x_0) = \delta(x-x_0)-G_\nu(x,x_0).
 \end{equation}Note that $G$ describes the absorption of photons at $x=x_0$. 
 Then \eqref{inu and itildenu} and \eqref{flow ansatz} together imply that the \textit{non-local heat equation} $\nabla_x\cdot \mathcal{F}=0$ becomes 
 \begin{multline*}
     0=\nabla_x \cdot \mathcal{F} = \int_0^\infty d\nu \int_\Omega dx_0\  \nabla_x\cdot \tilde{\mathcal{F}}_\nu(x;x_0)\alpha_\nu^a
     B_\nu(T(x_0))\\
     =\int_0^\infty d\nu \int_\Omega dx_0\  (\delta(x-x_0)-G_\nu(x,x_0))\alpha_\nu^a
     B_\nu(T(x_0))\\
     =\int_0^\infty d\nu\ \alpha_\nu^a 
     B_\nu(T(x)) - \int_0^\infty d\nu\int_\Omega dx_0\ G_\nu(x,x_0)\alpha_\nu^a
     B_\nu(T(x_0)).
 \end{multline*}
 We define 
 \begin{equation}\label{defw}w(x)\eqdef f(T(x))=\int_0^\infty d\nu\ \alpha_\nu^a
 B_\nu(T(x)),\end{equation} and $$J[w](x)\eqdef \int_0^\infty d\nu\int_\Omega dx_0\ G_\nu(x,x_0)\alpha_\nu^a
 B_\nu(f^{-1}(w(x_0))).$$Then, we have
$$J[w](x)=\int_0^\infty d\nu\int_\Omega dx_0\ G_\nu(x,x_0)\alpha_\nu^a
B_\nu(T(x_0)),$$ by \eqref{deff} and \eqref{defw}. 
 If we a priori have 
\begin{equation}\label{sufficient condition}\sup_{\nu>0}
\sup_{x\in \Omega}\int_\Omega G_\nu(x,x_0)dx_0\le \theta <  1,\end{equation} for some $0<\theta<1$, then we have 
 \begin{multline*}| J[w](x)| = \left|\int_0^\infty d\nu\ \int_\Omega dx_0\ G_\nu(x,x_0)\alpha_\nu^a
 B_\nu(T(x_0))\right|\\\le\|w(\cdot)\|_{L^\infty(\Omega)} \left|\sup_{\nu>0}\int_\Omega G_\nu(x,x_0)dx_0 \right|\le \theta\|w(\cdot)\|_{L^\infty(\Omega)},\end{multline*} by \eqref{defw} and \eqref{sufficient condition}. Therefore, $J$ is a mapping from $\mathcal{S} $ to itself where $$\mathcal{S}\eqdef \{w\in  L^\infty(\Omega): \|w\|_{L^\infty(\Omega)}\le L\},$$
 for some $L>0$ which is to be determined. Note that $\mathcal{S}$ is closed, bounded, convex, and non-empty. Furthermore, $J[w](\cdot)$ is equicontinuous, since we observe that
 \begin{multline}\label{equicontinuity}
     J[w](x)-J[w](y) \\= \int_0^\infty d\nu\int_\Omega dx_0\ (G_\nu(x,x_0)-G_\nu(y,x_0))\alpha_\nu^a
     B_\nu(f^{-1}(w(x_0)))\\
     =\int_0^\infty d\nu\int_{\mathbb{S}^2}dn\int_\Omega dx_0 \alpha_\nu^a\ (\tilde{I}_\nu(x,n;x_0)-\tilde{I}_\nu(y,n;x_0)) \alpha_\nu^a
     B_\nu(T(x_0))\\
     \le \int_0^\infty d\nu\  \alpha(\nu)\sup_{T> 0} B_\nu(\cdot)\int_{\mathbb{S}^2}dn(H_\nu(x,n)-H_\nu(y,n)) ,
 \end{multline} where $H_\nu$ will be defined as in \eqref{Hdef}, and $\int_{\mathbb{S}^2}dn\ H_\nu(\cdot,n)$ is (a priori) uniformly continuous on $\Omega$ by the Duhamel formula that we will obtain in \eqref{after iteration}. Thus, $J$ is compact by the Arzela-Ascoli theorem. Therefore, the Schauder fixed-point theorem yields that there exists a solution $w\in \mathcal{S}\subset L^\infty(\Omega).$  
 
 Therefore, it suffices to prove the uniformly-boundedness condition \eqref{sufficient condition} on $G_\nu$ and the uniform continuity of the kernel $\int_{\mathbb{S}^2}dn\ H_\nu(\cdot, n)$. If we define the term $H_\nu=H_\nu(x,n)$ as 
 \begin{equation}
     \label{Hdef}
     H_\nu(x,n)\eqdef \alpha^a_\nu
     \int_{\Omega} \tilde{I}_\nu(x,n;x_0)dx_0,
 \end{equation}then the condition \eqref{sufficient condition} is equivalent to 
 \begin{equation}
     \label{sufficient condition2}
     \sup_{\nu>0}\sup_{x\in\Omega}\int_{\mathbb{S}^2}H_\nu(x,n)dn\le \theta<1.
 \end{equation}Therefore, in the rest of this section, we will prove the equivalent condition \eqref{sufficient condition2} on $H_\nu$ and the uniform continuity of the kernel $\int_{\mathbb{S}^2}dn\ H_\nu(\cdot, n)$.  To this end, we multiply the Green-type equation \eqref{greentype eq} by $\alpha_\nu^a$, then integrate the equation on $\Omega$ with respect to $dx_0$,  and obtain the following equation for $H_\nu$:
 \begin{multline}
     \label{eq for Hnu}
     n\cdot\nabla_xH_\nu(x,n)=\frac{1}{4\pi}\alpha^a_\nu
     \chi_{\Omega}(x)+\alpha_\nu^s
     \int_{\mathbb{S}^2}K_\nu(n,n')H_\nu(x,n')dn'\\-(\alpha^a_\nu
     +\alpha^s_\nu
     )H_\nu(x,n).
 \end{multline}
 By integrating the equation \eqref{eq for Hnu} over $(x,n)\in \rth\times \mathbb{S}^2$ we obtain the following integral equation:
 \begin{multline}\label{Hnu 1}
     H_\nu(x,n)=
  \frac{1}{4\pi}  \int_0^{\infty}e^{-(\alpha^a_\nu+\alpha^s_\nu)\xi}\alpha_\nu^a \chi_\Omega(x-\xi n)d\xi \\+ \int_0^{\infty}e^{-(\alpha^a_\nu+\alpha^s_\nu)\xi}\alpha_\nu^s\int_{\mathbb{S}^2}K_\nu(n,n')H_\nu(x-\xi n,n')dn'd\xi .
 \end{multline} 
 Now, we define the operators 
 $\mathcal{T}$ and $\mathcal{U}$ as follows:
 \begin{equation}
     \label{TUdef}
     \begin{split}
     \mathcal{T}f(x,n)&\eqdef \frac{1}{4\pi}\int_0^{ \infty}e^{-(\alpha^a_\nu+\alpha^s_\nu)\xi} f(x-\xi n)d\xi, \\
     \mathcal{U}f(x,n)&\eqdef \int_0^{\infty}e^{-(\alpha^a_\nu+\alpha^s_\nu)\xi}\alpha_\nu^s\int_{\mathbb{S}^2}K_\nu(n,n')f(x-\xi n,n')dn'd\xi.
     \end{split}
 \end{equation}
 Then, we note that 
 $ \mathcal{T}= \mathcal{T}(\alpha_\nu^s,\alpha_\nu^a)$ and $$ \mathcal{T}(\alpha_\nu^a \chi_\Omega)\le \theta <1,$$ for a sufficiently small $\alpha_\nu^a $ for some $\theta\in(0,1)$ that depends on the geometry. Also, observe that 
 \eqref{Hnu 1} can now be written as 
 \begin{equation}\label{before iteration} H_\nu(x,n)= 
 \mathcal{T}(\alpha_\nu^a\chi_\Omega)(x,n)+ \mathcal{U}H_\nu(x,n).\end{equation}
 Then we repeat iterating this representation \eqref{before iteration} and obtain the following formal Duhamel series representation of $H_\nu$:
 \begin{multline}
     \label{after iteration}
     H_\nu(x,n)= \mathcal{T}(\alpha_\nu^a\chi_\Omega)(x,n)+  \mathcal{U}\mathcal{T}(\alpha_\nu^a\chi_\Omega)(x,n)+ \mathcal{U} \mathcal{U}\mathcal{T}(\alpha_\nu^a\chi_\Omega)(x,n)+\cdots\\
     =\sum_{j=0}^\infty (\mathcal{U})^j\mathcal{T}(\alpha_\nu^a\chi_\Omega)(x,n).
 \end{multline} 
We first check the convergence of the series.
 We show that $ \int_{\mathbb{S}^2} H_\nu dn $ is bounded, since each term in the series is non-negative. We first observe that
 \begin{multline}\label{T}
  \int_{\mathbb{S}^2}   \mathcal{T}(\alpha_\nu^a\chi_\Omega)dn= \frac{1}{4\pi}\int_{\mathbb{S}^2}\int_0^{ \infty}e^{-(\alpha^a_\nu+\alpha^s_\nu)\xi} \alpha_\nu^a \chi_\Omega (x-\xi n,n)d\xi dn\\
  = \frac{1}{4\pi}\int_{\mathbb{S}^2}\int_0^{ s(x,n)}e^{-(\alpha^a_\nu+\alpha^s_\nu)\xi} \alpha_\nu^a d\xi dn\\ = \frac{1}{4\pi}\int_{\mathbb{S}^2}\frac{\alpha_\nu^a}{\alpha^a_\nu+\alpha^s_\nu}(1-e^{-(\alpha^a_\nu+\alpha^s_\nu) s(x,n)})dn\le   \frac{\alpha_\nu^a}{\alpha^a_\nu+\alpha^s_\nu}(1-e^{-(\alpha^a_\nu+\alpha^s_\nu)D}),
 \end{multline}where $D$ is  the maximal diameter of $\Omega$.
 For the next term, we check that 
\begin{multline}\label{UT}
  \int_{\mathbb{S}^2}  \mathcal{U}\mathcal{T}(\alpha_\nu^a\chi_\Omega)dn= \frac{1}{4\pi}\int_{\mathbb{S}^2}\mathcal{U}\int_0^{ \infty}e^{-(\alpha^a_\nu+\alpha^s_\nu)\xi} \alpha_\nu^a \chi_\Omega (x-\xi n,n)d\xi dn\\
  =\frac{1}{4\pi} \int_{\mathbb{S}^2} \int_0^{ \infty}e^{-(\alpha^a_\nu+\alpha^s_\nu)\xi}\alpha_\nu^s\int_{\mathbb{S}^2}K_\nu(n,n')\\\times \int_0^{\infty}e^{-(\alpha^a_\nu+\alpha^s_\nu)\xi'} \alpha_\nu^a \chi_\Omega (x-\xi n-\xi' n',n')d\xi' dn'd\xi dn \\ 
  =\frac{1}{4\pi} \int_{\mathbb{S}^2} \int_0^{ s(x,n)}e^{-(\alpha^a_\nu+\alpha^s_\nu)\xi}\alpha_\nu^s\int_{\mathbb{S}^2}K_\nu(n,n')\int_0^{s(x-\xi n,n')}e^{-(\alpha^a_\nu+\alpha^s_\nu)\xi'} \alpha_\nu^a d\xi' dn'd\xi dn  \\\le  \frac{1}{4\pi} \int_{\mathbb{S}^2} \int_0^{ s(x,n)}e^{-(\alpha^a_\nu+\alpha^s_\nu)\xi}\alpha_\nu^s\int_{\mathbb{S}^2}K_\nu(n,n')\frac{\alpha_\nu^a}{\alpha^a_\nu+\alpha^s_\nu}(1-e^{-(\alpha^a_\nu+\alpha^s_\nu)D}) dn'd\xi dn \\\le \frac{1}{4\pi} \int_0^{D}e^{-(\alpha^a_\nu+\alpha^s_\nu)\xi}\alpha_\nu^s\int_{\mathbb{S}^2}\frac{\alpha_\nu^a}{\alpha^a_\nu+\alpha^s_\nu}(1-e^{-(\alpha^a_\nu+\alpha^s_\nu)D}) dn'd\xi \\=  \alpha^a_\nu \frac{\alpha_\nu^s}{(\alpha^a_\nu+\alpha^s_\nu)^2}(1-e^{-(\alpha^a_\nu+\alpha^s_\nu)D})^2 ,
 \end{multline} by \eqref{scatteringone}. Similarly, we repeat iterating and use $\chi_\Omega \le 1$ on the next iterated terms to have the following bound:
 \begin{equation}
     \label{Un bound}
      \int_{\mathbb{S}^2}  (\mathcal{U})^j\mathcal{T}(\alpha_\nu^a\chi_\Omega)dn\le   \alpha^a_\nu \frac{(\alpha_\nu^s)^{j}}{(\alpha^a_\nu+\alpha^s_\nu)^{j+1}}(1-e^{-(\alpha^a_\nu+\alpha^s_\nu)D})^{j+1}.
 \end{equation} Since the right-hand side of \eqref{Un bound} is in the form of a geometric series, we finally have
 \begin{multline}
     \label{Hnu bound}
     \int_{\mathbb{S}^2} H_\nu(x,n) dn \le \sum_{j=0}^\infty   \alpha^a_\nu \frac{(\alpha_\nu^s)^{j}}{(\alpha^a_\nu+\alpha^s_\nu)^{j+1}}(1-e^{-(\alpha^a_\nu+\alpha^s_\nu)D})^{j+1}\\=\frac{\frac{  \alpha^a_\nu}{\alpha^a_\nu+\alpha^s_\nu} (1-e^{-(\alpha^a_\nu+\alpha^s_\nu)D})}{1-\frac{\alpha_\nu^s}{\alpha^a_\nu+\alpha^s_\nu}(1-e^{-(\alpha^a_\nu+\alpha^s_\nu)D})}
     =\frac{  \alpha^a_\nu(1-e^{-(\alpha^a_\nu+\alpha^s_\nu)D})}{\alpha^a_\nu+\alpha_\nu^s e^{-(\alpha^a_\nu+\alpha^s_\nu)D}}<\infty,
 \end{multline}
by \eqref{after iteration}, since the ratio $\frac{\alpha_\nu^s}{\alpha^a_\nu+\alpha^s_\nu}(1-e^{-(\alpha^a_\nu+\alpha^s_\nu)D})<1$. Since the partial sum is monotonically increasing and the total sum is bounded, the series converges.  In addition, we also obtain that 
$$\sup_{\nu>0}\sup_{x\in\Omega}\int_{\mathbb{S}^2}H_\nu(x,n)dn\le \theta<1,$$
since $$\theta \eqdef \frac{  \alpha^a_\nu(1-e^{-(\alpha^a_\nu+\alpha^s_\nu)D})}{\alpha^a_\nu+\alpha_\nu^s e^{-(\alpha^a_\nu+\alpha^s_\nu)D}}<1.$$ This  proves  the a priori bound \eqref{sufficient condition2} and hence \eqref{sufficient condition} as well. 

Lastly, we check the uniform continuity of the kernel $\int_{\mathbb{S}^2}dn\ H_\nu(\cdot, n)$ to close the Schauder fixed-point theorem. We observe  \begin{multline*}
     \int_{\mathbb{S}^2}   \mathcal{T}(\alpha_\nu^a\chi_\Omega)dn=\frac{1}{4\pi} \int_{\mathbb{S}^2}\int_0^{ \infty}e^{-(\alpha^a_\nu+\alpha^s_\nu)\xi} \alpha_\nu^a \chi_\Omega (x-\xi n,n)d\xi dn\\
    = \frac{1}{4\pi}\int_{\mathbb{S}^2}\int_0^{s(x,n)}e^{-(\alpha^a_\nu+\alpha^s_\nu)\xi} \alpha_\nu^a d\xi dn,
 \end{multline*}so that
 \begin{multline}\label{diff1}
      \int_{\mathbb{S}^2}   \mathcal{T}(\alpha_\nu^a\chi_\Omega)(x,n)- \mathcal{T}(\alpha_\nu^a\chi_\Omega)(y,n)dn
      =\frac{1}{4\pi}\int_{\mathbb{S}^2}\int_{s(y,n)}^{s(x,n)}e^{-(\alpha^a_\nu+\alpha^s_\nu)\xi} \alpha_\nu^a d\xi dn\\=\frac{1}{4\pi}\int_{\mathbb{S}^2}\frac{\alpha_\nu^a}{\alpha^a_\nu+\alpha^s_\nu}(e^{-(\alpha^a_\nu+\alpha^s_\nu)s(y,n)}-e^{-(\alpha^a_\nu+\alpha^s_\nu)s(x,n)})dn\\
     \le \frac{1}{4\pi}\int_{\mathbb{S}^2}dn\ \alpha_\nu^a\int_0^1 d\tau \ e^{-(\alpha^a_\nu+\alpha^s_\nu)(\tau s(x,n)+(1-\tau)s(y,n))} |s(x,n)-s(y,n)|\\
     \le\frac{1}{4\pi}\int_{\mathbb{S}^2}dn\ \alpha_\nu^a \int_0^1d\tau' |\nabla_x s(\tau' x+(1-\tau')y,n)||x-y|\lesssim C_\Omega\alpha_\nu^a|x-y|,
 \end{multline}since the boundary of $\Omega$ is $C^1$. We also used the fundamental theorem of calculus. Thus, we obtain the uniform continuity of the first term in the series.
 
 Now, for the following term in the series, we observe that, by \eqref{UT},  \begin{multline}\label{diff2}
     \Bigg| \int_{\mathbb{S}^2}  \mathcal{U}\mathcal{T}(\alpha_\nu^a\chi_\Omega)(x,n)dn-  \int_{\mathbb{S}^2}  \mathcal{U}\mathcal{T}(\alpha_\nu^a\chi_\Omega)(y,n)dn\Bigg|\\
  =\Bigg|\frac{1}{4\pi} \int_{\mathbb{S}^2} \int_0^{ s(x,n)}e^{-(\alpha^a_\nu+\alpha^s_\nu)\xi}\alpha_\nu^s\int_{\mathbb{S}^2}K_\nu(n,n') \int_0^{s(x-\xi n,n')}e^{-(\alpha^a_\nu+\alpha^s_\nu)\xi'} \alpha_\nu^a d\xi' dn'd\xi dn\\
 - \frac{1}{4\pi} \int_{\mathbb{S}^2} \int_0^{ s(y,n)}e^{-(\alpha^a_\nu+\alpha^s_\nu)\xi}\alpha_\nu^s\int_{\mathbb{S}^2}K_\nu(n,n')\int_0^{s(y-\xi n,n')}e^{-(\alpha^a_\nu+\alpha^s_\nu)\xi'} \alpha_\nu^a d\xi' dn'd\xi dn\Bigg|\\
 \le \Bigg|\frac{1}{4\pi} \int_{\mathbb{S}^2} \int_{s(y,n)}^{ s(x,n)}e^{-(\alpha^a_\nu+\alpha^s_\nu)\xi}\alpha_\nu^s\int_{\mathbb{S}^2}K_\nu(n,n') \int_0^{D}e^{-(\alpha^a_\nu+\alpha^s_\nu)\xi'} \alpha_\nu^a d\xi' dn'd\xi dn\Bigg|\\
 +\Bigg|\frac{1}{4\pi} \int_{\mathbb{S}^2} \int_0^{ D}e^{-(\alpha^a_\nu+\alpha^s_\nu)\xi}\alpha_\nu^s\int_{\mathbb{S}^2}K_\nu(n,n')\int^{s(x-\xi n,n')}_{s(y-\xi n,n')}e^{-(\alpha^a_\nu+\alpha^s_\nu)\xi'} \alpha_\nu^a d\xi' dn'd\xi dn\Bigg|\\
 \eqdef I+II.
 \end{multline}Note that   \begin{multline}\label{diff3}
     I=\frac{1}{4\pi} \int_{\mathbb{S}^2}dn\int_{\mathbb{S}^2}dn'\ K_\nu(n,n')\frac{\alpha_\nu^a\alpha^s_\nu}{\alpha^a_\nu+\alpha^s_\nu}(1-e^{-(\alpha^a_\nu+\alpha^s_\nu)D})\Bigg| \int_{s(y,n)}^{ s(x,n)}e^{-(\alpha^a_\nu+\alpha^s_\nu)\xi}  d\xi\Bigg| \\
     \le \frac{1}{4\pi} \int_{\mathbb{S}^2}dn\int_{\mathbb{S}^2}dn'\ K_\nu(n,n')\frac{\alpha_\nu^a\alpha^s_\nu}{(\alpha^a_\nu+\alpha^s_\nu)^2}(1-e^{-(\alpha^a_\nu+\alpha^s_\nu)D})\\\times |e^{-(\alpha^a_\nu+\alpha^s_\nu)s(x,n)}-e^{-(\alpha^a_\nu+\alpha^s_\nu)s(y,n)}|\\
     \le \frac{1}{4\pi} \int_{\mathbb{S}^2}dn\int_{\mathbb{S}^2}dn'\ K_\nu(n,n')\frac{\alpha_\nu^a\alpha^s_\nu}{\alpha^a_\nu+\alpha^s_\nu}(1-e^{-(\alpha^a_\nu+\alpha^s_\nu)D})\\\times  \int_0^1 d\tau \ e^{-(\alpha^a_\nu+\alpha^s_\nu)(\tau s(x,n)+(1-\tau)s(y,n))} |s(x,n)-s(y,n)|\\
      \le \frac{1}{4\pi} \int_{\mathbb{S}^2}dn\int_{\mathbb{S}^2}dn'\ K_\nu(n,n')\frac{\alpha_\nu^a\alpha^s_\nu}{\alpha^a_\nu+\alpha^s_\nu}(1-e^{-(\alpha^a_\nu+\alpha^s_\nu)D})\\\times\int_0^1d\tau' |\nabla_x s(\tau' x+(1-\tau')y,n)||x-y| \\
     \le \frac{C_\Omega}{4\pi} \int_{\mathbb{S}^2}dn\int_{\mathbb{S}^2}dn'\ K_\nu(n,n')\frac{\alpha_\nu^a\alpha^s_\nu}{\alpha^a_\nu+\alpha^s_\nu}(1-e^{-(\alpha^a_\nu+\alpha^s_\nu)D})|x-y|\\
     =  \frac{C_\Omega\alpha_\nu^a\alpha^s_\nu}{\alpha^a_\nu+\alpha^s_\nu}(1-e^{-(\alpha^a_\nu+\alpha^s_\nu)D})|x-y|
   ,
 \end{multline} by the fundamental theorem of calculus and \eqref{partialsn} and that the boundary of $\Omega$ is $C^1$. In addition, we also note that
  \begin{multline}\label{diff4}
     II\le \frac{1}{4\pi} \int_{\mathbb{S}^2} \int_0^{ D}e^{-(\alpha^a_\nu+\alpha^s_\nu)\xi}\alpha_\nu^s\int_{\mathbb{S}^2}K_\nu(n,n')\frac{\alpha_\nu^a}{\alpha^a_\nu+\alpha^s_\nu}\\\times |e^{-(\alpha^a_\nu+\alpha^s_\nu)s(x-\xi n,n')}-e^{-(\alpha^a_\nu+\alpha^s_\nu)s(y-\xi n,n')}|dn'd\xi dn\\
     \le \frac{1}{4\pi} \int_{\mathbb{S}^2} \int_0^{ D}e^{-(\alpha^a_\nu+\alpha^s_\nu)\xi}\alpha_\nu^s\int_{\mathbb{S}^2}K_\nu(n,n')\ \alpha_\nu^a\\\times \int_0^1 d\tau \ e^{-(\alpha^a_\nu+\alpha^s_\nu)(\tau s(x-\xi n,n')+(1-\tau)s(y-\xi n,n'))} |s(x-\xi n,n')-s(y-\xi n,n')|dn'd\xi dn\\
     \le\frac{1}{4\pi} \int_{\mathbb{S}^2} \int_0^{ D}e^{-(\alpha^a_\nu+\alpha^s_\nu)\xi}\alpha_\nu^s\int_{\mathbb{S}^2}K_\nu(n,n')\ \alpha_\nu^a \\\times \int_0^1d\tau' |\nabla_x s(\tau' x+(1-\tau')y-\xi n,n')||x-y|dn'd\xi dn\\
    \le \frac{C_\Omega}{4\pi} \int_{\mathbb{S}^2} \int_0^{ D}e^{-(\alpha^a_\nu+\alpha^s_\nu)\xi}\alpha_\nu^s\int_{\mathbb{S}^2}K_\nu(n,n')\ \alpha_\nu^a |x-y|dn'd\xi dn\\= \frac{C_\Omega}{4\pi}  \int_0^{ D}e^{-(\alpha^a_\nu+\alpha^s_\nu)\xi}\alpha_\nu^s \alpha_\nu^a |x-y|d\xi =C_\Omega\alpha^a_\nu \frac{\alpha_\nu^s}{\alpha^a_\nu+\alpha^s_\nu}(1-e^{-(\alpha^a_\nu+\alpha^s_\nu)D})|x-y|,
 \end{multline} by the fundamental theorem of calculus and  \eqref{partialsn} and that the boundary of $\Omega$ is $C^1$. Thus, we obtain the uniform continuity of the second term in the series $\int_{\mathbb{S}^2}\mathcal{U}\mathcal{T}(\alpha_\nu^a\chi_\Omega)(\cdot,n)dn$ as 
 \begin{multline}\notag
     \Bigg| \int_{\mathbb{S}^2}  \mathcal{U}\mathcal{T}(\alpha_\nu^a\chi_\Omega)(x,n)dn-  \int_{\mathbb{S}^2}  \mathcal{U}\mathcal{T}(\alpha_\nu^a\chi_\Omega)(y,n)dn\Bigg|\\
    \le  2C_\Omega\alpha^a_\nu \left(\frac{\alpha_\nu^s}{\alpha^a_\nu+\alpha^s_\nu}\right)(1-e^{-(\alpha^a_\nu+\alpha^s_\nu)D})|x-y|.
     \end{multline}
 
 Similarly, we repeat the same estimates for all other iterated elements and obtain 
 \begin{multline}\label{diff5}
     \Bigg| \int_{\mathbb{S}^2}  \mathcal{U}^k\mathcal{T}(\alpha_\nu^a\chi_\Omega)(x,n)dn-  \int_{\mathbb{S}^2}  \mathcal{U}^k\mathcal{T}(\alpha_\nu^a\chi_\Omega)(y,n)dn\Bigg|\\
    \le  C_\Omega(k+1)\alpha^a_\nu \left(\frac{\alpha_\nu^s}{\alpha^a_\nu+\alpha^s_\nu}\right)^k(1-e^{-(\alpha^a_\nu+\alpha^s_\nu)D})^k|x-y|,
     \end{multline}since the $(k+1)^{\text{th}}$ term will need the decomposition of the integral domain into $(k+1)$ parts. 
     Therefore, we take the series sum and finally obtain the upper bound as
     \begin{equation*}
     \Bigg|\int_{\mathbb{S}^2}dn(H_\nu(x,n)-H_\nu(y,n)) \Bigg|
    \le  C_\Omega\alpha^a_\nu\frac{1}{\left(1-\frac{\alpha_\nu^s}{\alpha^a_\nu+\alpha^s_\nu}(1-e^{-(\alpha^a_\nu+\alpha^s_\nu)D})\right)^2} |x-y|,
     \end{equation*}where we used the formula that
     $$\sum_{k=0}^\infty (k+1)\theta^k = \frac{1}{(1-\theta)^2}.$$
 This proves the desired uniform continuity of the kernel $\int_{\mathbb{S}^2}dn\ H(\cdot,n)$ for the compactness of the map $J$ for the Arzela-Ascoli theorem at \eqref{equicontinuity}. 
 This completes the proof of the existence of solutions to \eqref{full system} coupled with \eqref{heat equation} and \eqref{incoming boundary}. 
 \end{proof}
  

\section*{Acknowledgement} The authors gratefully acknowledge the support of the grant CRC
1060 ``The Mathematics of Emergent Effects" of the University of Bonn funded
through the Deutsche Forschungsgemeinschaft (DFG, German Research Foundation). J. W. Jang is supported by the National Research Foundation of Korea (NRF) grants funded by the Korean government (MSIT) NRF-2022R1G1A1009044 and No. RS-2023-00210484. J. W. Jang is also supported by the Basic Science Research Institute Fund of Korea NRF-2021R1A6A1A10042944. Juan J. L. Vel\'azquez is also funded by DFG under Germany's Excellence Strategy-EXC-2047/1-390685813.

\bibliographystyle{amsplain3links}
\bibliography{bibliography.bib}{}

\providecommand{\bysame}{\leavevmode\hbox to3em{\hrulefill}\thinspace}
\providecommand{\href}[2]{#2}
\begin{thebibliography}{10}
\expandafter\ifx\csname urlpdf\endcsname\relax
  \def\urlpdf#1{\url{#1}}\fi
\expandafter\ifx\csname arxiv\endcsname\relax
  \def\arxiv#1{\burlalt{arXiv:#1}{http://arxiv.org/abs/#1}}\fi
\expandafter\ifx\csname doi\endcsname\relax
  \def\doi#1{\burlalt{doi:#1}{http://dx.doi.org/#1}}\fi
\expandafter\ifx\csname href\endcsname\relax
  \def\href#1#2{#2}\fi
\expandafter\ifx\csname burlalt\endcsname\relax
  \def\burlalt#1#2{\href{#2}{#1}}\fi

\bibitem{B1}
C.~Bardos, F.~Golse, and B.~Perthame, \emph{The radiative transfer equations:
  existence of solutions and diffusion approximation under accretivity
  assumptions---a survey}, Proceedings of the conference on mathematical
  methods applied to kinetic equations ({P}aris, 1985), vol.~16, 1987,
  pp.~637--652, \doi{10.1080/00411458708204308}.

\bibitem{B3}
\bysame, \emph{The {R}osseland approximation for the radiative transfer
  equations}, Comm. Pure. Appl. Math. \textbf{40} (1987), no.~6, 691--721,
  \arxiv{https://onlinelibrary.wiley.com/doi/pdf/10.1002/cpa.3160400603},
  \doi{https://doi.org/10.1002/cpa.3160400603}.

\bibitem{Golse-Bardos}
C.~Bardos, F.~Golse, B.~Perthame, and R.~Sentis, \emph{The nonaccretive
  radiative transfer equations: existence of solutions and {R}osseland
  approximation}, J. Funct. Anal. \textbf{77} (1988), no.~2, 434--460,
  \doi{10.1016/0022-1236(88)90096-1}.

\bibitem{MR3324151}
G.~Basile, A.~Nota, F.~Pezzotti, and M.~Pulvirenti, \emph{Derivation of the
  {F}ick's law for the {L}orentz model in a low density regime}, Comm. Math.
  Phys. \textbf{336} (2015), no.~3, 1607--1636,
  \doi{10.1007/s00220-015-2306-z}.

\bibitem{MR533346}
A.~Bensoussan, J.-L. Lions, and G.~C. Papanicolaou, \emph{Boundary layers and
  homogenization of transport processes}, Publ. Res. Inst. Math. Sci.
  \textbf{15} (1979), no.~1, 53--157, \doi{10.2977/prims/1195188427}.

\bibitem{Compton}
K.~T. Compton, \emph{{LXXIII}. {S}ome properties of resonance radiation and
  excited atoms}, The London, Edinburgh, and Dublin Philosophical Magazine and
  Journal of Science \textbf{45} (1923), no.~268, 750--760.

\bibitem{MR3412332}
R.~M.~Saldanha da~Gama, \emph{Existence, uniqueness and construction of the
  solution of the energy transfer problem in a rigid and non-convex blackbody
  with temperature-dependent thermal conductivity}, Z. Angew. Math. Phys.
  \textbf{66} (2015), no.~5, 2921--2939, \doi{10.1007/s00033-015-0549-3}.

\bibitem{2208.04212}
E.~Demattè, \emph{On a kinetic equation describing the behavior of a gas
  interacting mainly with radiation}, 2022, \arxiv{arXiv:2208.04212}.

\bibitem{MR2478911}
P.-\'{E}. Druet, \emph{Weak solutions to a stationary heat equation with
  nonlocal radiation boundary condition and right-hand side in {$L^p\ (p\geq
  1)$}}, Math. Methods Appl. Sci. \textbf{32} (2009), no.~2, 135--166,
  \doi{10.1002/mma.1029}.

\bibitem{MR2600939}
\bysame, \emph{Weak solutions to a time-dependent heat equation with nonlocal
  radiation boundary condition and arbitrary {$p$}-summable right-hand side},
  Appl. Math. \textbf{55} (2010), no.~2, 111--149,
  \doi{10.1007/s10492-010-0005-9}.

\bibitem{MR3797032}
M.~Ghattassi, J.~R. Roche, and D.~Schmitt, \emph{Existence and uniqueness of a
  transient state for the coupled radiative-conductive heat transfer problem},
  Comput. Math. Appl. \textbf{75} (2018), no.~11, 3918--3928,
  \doi{10.1016/j.camwa.2018.03.002}.

\bibitem{golse1987milne}
F.~Golse, \emph{The {M}ilne problem for the radiative transfer equations (with
  frequency dependence)}, Transactions of the American Mathematical Society
  \textbf{303} (1987), no.~1, 125--143.

\bibitem{golse1986generalized}
F.~Golse and B.~Perthame, \emph{Generalized solutions of the radiative transfer
  equations in a singular case}, Communications in Mathematical Physics
  \textbf{106} (1986), no.~2, 211--239.

\bibitem{CRMECA_2022__350_S1_A12_0}
F.~Golse and O.~Pironneau, \emph{Stratified radiative transfer for
  multidimensional fluids}, Reports. Mechanical (2022) (en),
  \doi{10.5802/crmeca.136}.

\bibitem{MR4519710}
\bysame, \emph{Radiative transfer in a fluid}, Rev. R. Acad. Cienc. Exactas
  F\'{\i}s. Nat. Ser. A Mat. RACSAM \textbf{117} (2023), no.~1, Paper No. 37,
  17, \doi{10.1007/s13398-022-01362-x}.

\bibitem{golse2008radiative}
F.~Golse and F.~Salvarani, \emph{Radiative transfer equations and {R}osseland
  approximation in gray matter}, Waves and Stability in Continuous Media, World
  Scientific, 2008, pp.~321--326.

\bibitem{MR3686005}
Y.~Guo and L.~Wu, \emph{Geometric correction in diffusive limit of neutron
  transport equation in 2{D} convex domains}, Arch. Ration. Mech. Anal.
  \textbf{226} (2017), no.~1, 321--403, \doi{10.1007/s00205-017-1135-y}.

\bibitem{Holstein}
T.~Holstein, \emph{Imprisonment of resonance radiation in gases}, Physical
  Review \textbf{72} (1947), no.~12, 1212.

\bibitem{2109.10071}
J.~W. Jang and J.~J.~L. Vel\'{a}zquez, \emph{L{TE} and non-{LTE} solutions in
  gases interacting with radiation}, J. Stat. Phys. \textbf{186} (2022), no.~3,
  Paper No. 47, 62, \doi{10.1007/s10955-022-02888-5}.

\bibitem{PhysRev.42.823}
C.~Kenty, \emph{On radiation diffusion and the rapidity of escape of resonance
  radiation from a gas}, Phys. Rev. \textbf{42} (1932), 823--842,
  \doi{10.1103/PhysRev.42.823}.

\bibitem{MR1608072}
M.~T. Laitinen and T.~Tiihonen, \emph{Integro-differential equation modelling
  heat transfer in conducting, radiating and semitransparent materials}, Math.
  Methods Appl. Sci. \textbf{21} (1998), no.~5, 375--392,
  \doi{10.1002/(SICI)1099-1476(19980325)21:5<375::AID-MMA953>3.0.CO;2-U}.

\bibitem{MR1866555}
\bysame, \emph{Conductive-radiative heat transfer in grey materials}, Quart.
  Appl. Math. \textbf{59} (2001), no.~4, 737--768, \doi{10.1090/qam/1866555}.

\bibitem{mihalas2013foundations}
D.~Mihalas and B.~W. Mihalas, \emph{Foundations of radiation hydrodynamics},
  Courier Corporation, 2013.

\bibitem{Milne}
E.~A. Milne, \emph{The diffusion of imprisoned radiation through a gas}, J.
  London Math. Soc. \textbf{1} (1926), no.~1, 40--51,
  \doi{10.1112/jlms/s1-1.1.40}.

\bibitem{MR3356368}
A.~Nota, \emph{Diffusive limit for the random {L}orentz gas}, From particle
  systems to partial differential equations. {II}, Springer Proc. Math. Stat.,
  vol. 129, Springer, Cham, 2015, pp.~273--292,
  \doi{10.1007/978-3-319-16637-7\_10}.

\bibitem{Nouri2}
A.~Nouri, \emph{Stationary states of a gas in a radiation field from a kinetic
  point of view}, Ann. Fac. Sci. Toulouse Math. (6) \textbf{10} (2001), no.~2,
  361--390.

\bibitem{oxenius}
J.~Oxenius, \emph{Kinetic theory of particles and photons}, Springer Series in
  Electrophysics, vol.~20, Springer-Verlag, Berlin, 1986,
  \doi{10.1007/978-3-642-70728-5}, Theoretical foundations of non-LTE plasma
  spectroscopy.

\bibitem{MR2076780}
M.~M. Porzio and \'{O}. L\'{o}pez-Pouso, \emph{Application of accretive
  operators theory to evolutive combined conduction, convection and radiation},
  Rev. Mat. Iberoamericana \textbf{20} (2004), no.~1, 257--275,
  \doi{10.4171/RMI/388}.

\bibitem{RSM}
A.~Rossani, G.~Spiga, and R.~Monaco, \emph{Kinetic approach for two-level atoms
  interacting with monochromatic photons}, Mechanics Research Communications
  \textbf{24} (1997), no.~3, 237--242.

\bibitem{rutten1995radiative}
R.~J. Rutten, \emph{Radiative transfer in stellar atmospheres}, Utrecht
  University lecture notes, 8th edition, 2003,
  \href{https://robrutten.nl/rrweb/rjr-pubs/2003rtsa.book.....R.pdf}{https://robrutten.nl/rrweb/rjr-pubs/2003rtsa.book.....R.pdf}.

\bibitem{spiegel1957smoothing}
E.~A. Spiegel, \emph{The smoothing of temperature fluctuations by radiative
  transfer.}, The Astrophysical Journal \textbf{126} (1957), 202.

\bibitem{MR1471600}
T.~Tiihonen, \emph{A nonlocal problem arising from heat radiation on non-convex
  surfaces}, European J. Appl. Math. \textbf{8} (1997), no.~4, 403--416,
  \doi{10.1017/S0956792597003185}.

\bibitem{MR3324149}
L.~Wu and Y.~Guo, \emph{Geometric correction for diffusive expansion of steady
  neutron transport equation}, Comm. Math. Phys. \textbf{336} (2015), no.~3,
  1473--1553, \doi{10.1007/s00220-015-2315-y}.

\end{thebibliography}

 \end{document}